\documentclass{amsart}
\usepackage{amsmath,amssymb}
\usepackage{amsfonts}
\usepackage{amsthm}
\usepackage{xcolor}
\newtheorem{prop}[subsection]{Proposition}
\newtheorem{corollary}[subsection]{Corollary}
\newtheorem{theorem}[subsection]{Theorem}
\newtheorem{definition}[subsection]{Definition}
\newtheorem{lemma}[subsection]{Lemma}
\newtheorem{example}[subsection]{Example}
	\date{}

\begin{document}
	\title{Left semi-braces and solutions of the Yang-Baxter equation
}

	\author{Eric Jespers \and Arne Van Antwerpen}
	\address{Department of Mathematics, Vrije Universiteit Brussel\\
	Pleinlaan 2, 1050 Brussel}
	\email{
	eric.jespers@vub.be and arne.van.antwerpen@vub.be}

		\thanks{The first author is supported in part
        by  Onderzoeksraad of Vrije Universiteit Brussel and 
        Fonds voor Wetenschappelijk Onderzoek (Belgium). The second author is supported by
        Fonds voor Wetenschappelijk Onderzoek (Vlaanderen).\\
        }
\keywords{ Yang-Baxter equation, semi-brace, monoid algebras, monoids, groups; 
MSC Codes: 20M25, 16T25, 20E22 
and 16S36}

	\maketitle
	
\begin{abstract}
 Let $r:X^{2}\rightarrow X^{2}$ be a set-theoretic solution of the Yang-Baxter 
 equation on a finite set $X$.
It was proven by Gateva-Ivanova and Van den Bergh that if $r$ is non-degenerate 
and involutive then the 
algebra $K\langle x \in X \mid xy =uv \mbox{ if } r(x,y)=(u,v)\rangle$ shares 
many properties with commutative
polynomial algebras in finitely many variables; in particular this algebra is 
Noetherian, satisfies a polynomial identity
and has Gelfand-Kirillov dimension a positive integer. Lebed and Vendramin 
recently extended this result to arbitrary
non-degenerate bijective solutions. Such solutions are naturally associated to 
finite skew left braces. In this paper
we will prove an analogue result for arbitrary solutions $r_B$ that are 
associated to a left semi-brace $B$; such solutions
can be degenerate or can  even be idempotent. In order to do so we first 
describe
such semi-braces and we prove some decompositions results extending results of  
Catino, Colazzo, and Stefanelli.
\end{abstract}

	\section{Introduction}

 The Yang-Baxter equation is an important tool in several fields of research,
 among these are statistical mechanics, particle physics, quantum field theory 
 and quantum group theory. We refer to \cite{backgroundYBE} for a brief 
 introduction. Studying the solutions of the Yang-Baxter equation has been a 
 major research area for the past $50$ years. In 1992, V. Drinfeld focused 
 attention on the so-called set-theoretic solutions, or braided sets. 
 Set-theoretic solutions 
 of the Yang-Baxter equation are sets $X$ with a map $r: X \times X 
 \longrightarrow X \times X$  such that, on $X^{3}$  the following equation
 is satisfied $$ \left( r 
 \times \textup{id}\right) \left( \textup{id}\times r \right) \left( r \times 
 \textup{id}\right) = \left( \textup{id} \times r \right) \left( r \times 
 \textup{id} \right) \left( \textup{id} \times r \right).$$
 Let $(X,r)$ be such a solution. If, furthermore, $r^2 = \textup{id}_{X\times X}$ then the solution is said to be involutive. 
 For  $x,y \in X$ define the maps $\sigma_x: X \longrightarrow X$ and 
 $\gamma_y: X \longrightarrow X$ by $r(x,y) = 
 \left(\sigma_x(y),\gamma_y(x)\right)$. A solution $(X,r)$ is called left 
 (respectively right) non-degenerate if, for any $x \in X$, $\sigma_x$ (respectively $\gamma_x$) is 
 bijective. A solution is called non-degenerate if it is both left and right 
 non-degenerate. In \cite{rump2005decomposition} Rump showed that every finite 
 involutive, non-degenerate set-theoretic solution of the Yang-Baxter equation 
 corresponds to a new algebraic structure, called a left brace. In 
 \cite{cedo2014braces} Ced\'o, Jespers and Okni\'nski showed that Rump's left 
 braces are equivalent to the currently often used definition. A left brace is a set 
 $A$ with an abelian group structure $(A,+)$ and a group structure 
 $(A,\circ)$ such that $\left(a \circ (b + 
 c)\right) + a = (a \circ b) + (a \circ c)$,  for any $a,b,c \in A$. Right braces are defined 
 similarly.  
 
 Left braces turn out to be a very useful tool for the investigations in different topics. We mention three such topics.
First,  in \cite{rump2005decomposition} Rump showed that two-sided braces (left braces that also are a right brace, for the same operations)
 are equivalent with radical rings. Moreover, through the work of 
 Gateva-Ivanova, Van den Bergh \cite{GatevaVandenBergh} and Etingof, Schedler 
 and Soloviev \cite{etingof1998set} the theory of braces was connected with the 
 theory of finitely generated quadratic algebras and monoids and groups of $I$-type. 
 If $(X,r)$ is an involutive non-degenerate set-theoretic solution on a finite set $X$ then the  monoid $M=M(X,r)=\langle x\in X \mid xy=uv \mbox{ if } r(x,y)=(u,v)
 \rangle$ is  called a monoid of $I$-type. In \cite{GatevaVandenBergh} it is 
 shown that, for any field $K$, the monoid algebra $KM$ shares many properties 
 with polynomial algebras in finitely many commuting generators. In particular, 
 these algebras are left and right Noetherian domain that satisfy a polynomial 
 identity. Furthermore, $M$ is embedded in its group of fractions 
 $G=G(X,r)=\mbox{gr}(x\in X \mid xy=uv \mbox{ if }  r(x,y)=(u,v))$, a solvable 
 group that is Bieberbach, i.e. a finitely generated, torsion free 
 abelian-by-finite group. Note that $G(X,r)$ also is called the structure group 
 of $(X,r)$
 in \cite{etingof1998set}). By analogy, we call $M(X,r)$ the structure monoid of $(X,r)$.
 Second, interest from group theory 
 stems from the equivalence of left braces and regular subgroups of the 
 holomorph of an abelian group and Hopf-Galois extensions as studied by Ced\'o, 
 Jespers and del R\'{\i}o 
 \cite{cedo2010involutive}, Goffa and Jespers \cite{goffa2007monoids}, Catino, 
 Colazzo and Stefanelli \cite{catino2016regular} and Gateva-Ivanova, Jespers 
 and Okni\'nski \cite{GIJO}. Third, braces are a useful tool for calculations in 
 ring theory and group theory, as shown by Smoktunowicz in 
 \cite{smoktunowicz2015engel}, where questions on Engel groups and nil 
 algebras, proposed by Sysak, Amberg and Zelmanov, are answered by calculations 
 in braces. 
 
 Guarnieri and Vendramin \cite{guarnieri2017skew} have shown that the study of 
 finite non-degenerate set-theoretic solutions of the Yang-Baxter 
 equation (i.e. not necessarily involutive) is equivalent with the study of 
 skew left braces,  a generalization of left braces. A skew 
 left brace is a set $B$ with two group structures $(B,\cdot)$ and $(B,\circ)$ 
 such that 
  \begin{eqnarray} a \circ (b\cdot c) &=& \left(a \circ b 
 \right)\cdot  a^{-1} \cdot  \left( a \circ c\right), \label{defskewbrace}
 \end{eqnarray} 
 for any $a,b,c \in B$, where $a^{-1}$ denotes the inverse of $a$ in 
 $(A,\cdot)$.  A skew right brace is defined 
 similarly. Furthermore, skew left braces are equivalent to regular subgroups 
 of the holomorph of a group, which can be related to the theory of Hopf-Galois 
 extensions as shown by Bachiller \cite{bachiller2016counterexample} (for 
 braces) and explained in an appendix by Byott and Vendramin 
 \cite{smoktunowicz2017skew}.
Of course, also in this case one can define the structure monoid and structure group of the non-degenerate solution $(X,r)$. Note, however,
that $X$ is not necessarily embedded anymore  in $G(X,r)$ (see for example 
\cite[3.4]{smoktunowicz2017skew}). The algebraic structure of 
$KM(X,r)$  is yet unknown  and in \cite[Thoerem 5.6]{lebed2017structure} it is 
shown that $G(X,r)$ again is free abelian-by-finite and
thus $KG(X,r)$ is Noetherian and satisfies a polynomial identity. Moreover, its Gelfand-Kirillov dimension is at most
$|X|$. 

In all the above, set-theoretic solutions $(X,r)$  are considered with $r$ a bijective mapping.
Lebed in \cite{lebedfactor} has shown that solutions $(X,r)$ with $r^{2}=r$, as well as their associated algebras $KM(X,r)$, also are of importance. They provide a powerful unifying tool, simultaneously treating very different algebraic structures, such as  free and free commutative monoids, factorizable monoids, distributive lattices and  Young tableaux and plactic monoids.
Such idempotent solutions also show up in  recent work of Catino, Colazzo and 
Stefanelli \cite{catino2017semi} on the investigations of left semi-braces 
$(B,\cdot, \circ)$, a generalisation of skew left braces. Here $(B,\circ) $ is a group and $(B,\cdot )$ is a left 
cancellative semigroup and both operations satisfy a compatibility condition 
that generalizes (\ref{defskewbrace}) (see definition). 

In this paper we  will investigate  arbitrary left semi-braces $(B,\cdot, \circ)$, i.e. the semigroup $(B,\cdot )$ is not necessarily left cancellative. 
We will describe (finite) left semi-braces in terms of skew left braces  and a left semi-brace with a 
multiplicative semigroup consisting of idempotents only; in particular they turn out to be a matrix 
semigroup over a skew left brace $G$. Under some (mild) assumption we will describe   left 
semi-braces as  a matched product of two twosided semi-braces and a skew left brace. 
Next, it is shown that to a finite left semi-brace $B$ one associates a natural set-theoretic solution $r_{B}$ on $B$. The above  
information will then be used to obtain algebraic information on the structure monoid $M(B,r_{B})$, a finitely presented monoid on $|B|$ generators satisfying homogeneous quadratic relations determined by $r_{B}$.
As an application it is shown that the monoid algebra $KM(B,r_B)$ (the 
algebra determined by the ``same'' presentation as $(B,r_{B})$, is a finite 
(left 
and right) module over $KM(G,r_G)$. Hence, it follows 
that the  algebra
$KM(B,r_B)$ is again (left and right) Noetherian, has  finite Gelfand-Kirillov 
dimension and satisfies a polynomial identity. We show that $r_{B}^{3}=r_{B}$ 
and, 
if $|G|=1$  then $r_{B}^{2}=r_{B}$.
Furthermore, we show that ideals of $B$, i.e. kernels $\ker (f) =\{ b \in B \mid f(b) =1_{\circ} \}$ of left semi-brace homomorphisms $f$, are precisely the normal subgroups $(I,\circ)$ of $(B,\circ)$ such that $(I\cap G,\cdot )$ is a normal subgroup of $(G,\cdot)$ and that are invariant under some natural maps associated to $B$. An example of such an ideal is the socle; a tool that has shown its importance for (skew) left braces. 

\section{A description of left semi-braces}

We begin by defining the notion of a left semi-brace $(B, \cdot , \circ )$. Our 
definition extends the one introduced by F. Catino, I. Colazzo and P. 
Stefanelli \cite{catino2017semi} where one works under the restriction that  
the semigroup $(B,\cdot )$ is left cancellative.

\begin{definition}\label{skewbrace} Let $B$ be a set with two operations 
$\cdot$ and $\circ$ such that 
$(B,\cdot)$ is a semigroup and $(B, \circ)$ is a group. 
One says that $(B, \cdot , \circ )$ is left semi-brace if 	       
        $$ a \circ (b\cdot c) = (a \circ b) \cdot \left( a \circ ( \overline{a} 
        \cdot c )\right),$$
for all $a,b,c \in B$. 
Here, $\overline{a}$ denotes the inverse of $a$ in $(B,\circ)$. 
We call $(B,\cdot )$ the multiplicative semigroup of the left semi-brace of $(B,\cdot , \circ )$.
If the semi-group $(B,\cdot)$ has a pre-fix, pertaining to some property of the semi-group, we will also use this pre-fix with the left semi-brace.

If, moreover,  $(B,\cdot )$ is group then $(B, \cdot , \circ )$ is called a 
skew left brace (introduced by  Guarneri and Vendramin in \cite{guarnieri2017skew}). If $(B,\cdot )$ 
is an abelian group then $(B, \cdot , \circ )$ is called a brace (introduced 
by  Rump in \cite{RumpBraces}).
\end{definition}

Throughout we denote  by $1_{\circ}$ the identity of the group $(B, \circ)$. 
For $a,b\in B$ we denote the product $a\cdot b$ simply as $ab$. 

Recall that a semigroup   $(S,\cdot)$ is said to be a right zero semigroup if  
 $xy = y$ for any $x,y \in S$. Hence,   a left semi-brace $(B,\cdot,\circ)$, where $(B,\cdot)$ is a right 
zero semigroup is called a right zero left semi-brace.
For more terminology and background on semigroup theory we refer the reader to 
 \cite{Howie}. The subset consisting of idempotents in $S$ is denoted by 
 $E(S)$. In particular, by $E(B)$ we denote the
 idempotents of the semigroup $(B,\cdot )$ of a left semi-brace $(B,\cdot,\circ)$.

We recall the characterization of left cancellative left semi-braces  obtained 
by F. Catino, I. Colazzo and P. Stefanelli \cite{catino2017semi}.
Of course, a map $f: B_1 	\longrightarrow B_2$ between two left semi-braces $(B_1,\cdot,\circ)$ and $(B_2,\cdot,\circ)$ is said to be a left semi-brace homomorphsim if
it respects both  operations $\cdot$ and $\circ$.

\begin{theorem}\cite{catino2017semi} \label{Colazzochar}
If $(B,\cdot, \circ)$ is a left cancellative left semi-brace
then $(B, \cdot) \cong B 1_{\circ} \times E(B)$, and $(B1_{\circ},\cdot) $ is a group. Furthermore, $(B 
1_{\circ}, \cdot , \circ)$ is a skew left brace, $(E(B),\cdot, \circ)$ is a right zero 
left semi-brace and  $(B, \cdot, \circ) \cong (B  1_{\circ}, \cdot, \circ) 
\bowtie (E(B), \cdot, \circ)$,  a matched product of semi-braces (see Definition 
\ref{Matchedproductsemibrace}). Conversely, any such matched product is a left 
cancellative left semi-brace. 
\end{theorem}

Using the notation of the theorem, it can be seen that the multiplicative semigroup  of $B$  is isomorphic with the multiplicative completely $0$-simple semigroup
$\mathcal{M}(G,1,|E(B)|,\mathcal{I}_{|E(B)|,1})$, where $(G,\cdot)$ is the 
group $(B1_{\circ}, \cdot )$ and  $\mathcal{I}_{i,k}$ 
denotes the $i\times k-$matrix with $1$ in every entry (here $1$ denotes the 
identity of the group $G$). 


In this section we show that the multiplicative semigroup of an arbitrary finite left semi-brace is a   completely simple semigroup, that is (see
Theorem 3.3.1 in \cite{Howie})	
$(B,\cdot )$ is isomorphic with a matrix semigroup
$\mathcal{M}(G, I,J,P)=\{ (g,i,j)\mid g\in G,\; i\in I, \; j\in J\}$ for some group $G$, sets  $I,J$ 
and $J\times I$-matrix $P$ with entries in $G$. The multiplication is as  follows: $(g,i,j) (h,k,l) =(gp_{jk}h,i,l)$. Furthermore, we show that
one may take $=\mathcal{I}_{|J|,|I|}$.
%
A counterexample to show that this does not 
hold for infinite left semi-braces is not known to the authors.

We begin by showing left semi-braces $B$  do not contain a zero element if $|B|\geq 2$.

\begin{lemma} \label{nozero}
If $(B,\cdot,\circ)$ is a left semi-brace with at least two elements then the multiplicative semigroup of $B$ does not contain a zero element.
\end{lemma}

\begin{proof}
Suppose  $\theta$ is a zero element of $(B,\cdot )$. 
Then, for  $a, b \in B$, 
   $$ a \circ \theta = a \circ ( b \theta) = (a \circ b) \left( a \circ ( \overline{a}\theta)\right) 
     = \left(a \circ b \right) \left( a \circ \theta\right).$$
Hence, for $b=\overline{a} \circ \theta$ we get that
 $a\circ \theta = \theta (a \circ \theta ) =\theta $,
for all $a\in B$.
As $(B,\circ )$ is a group, it follows that $B=\{ \theta \}$, a contradiction.
\end{proof}

Recall that if $S$ is a semigroup then $S^{1}$ denotes the smallest monoid containing $S$.

\begin{lemma} \label{S2isS1S}
Let $(B,\cdot,\circ)$ be a left semi-brace. 
The following properties hold.
\begin{enumerate}
\item[(1)]
$ab=a1_{\circ} b$ for all $a,b\in B$. 
In particular, $B^{2}= B1_{\circ}B$ and $B^{1} 1_{\circ} B^{1} =
B1_{\circ} B \cup \{ 1_{\circ} \}$.
\item[(2)]  $B^{1}1_{\circ}B^{1}$ is a subgroup of the multiplicative semigroup $(B,\circ)$.
In particular, $(B^{1}1_{\circ}B^{1}, \cdot , \circ)$ is a sub semi-brace of $(B,\cdot , \circ)$.
\end{enumerate}
\end{lemma}

\begin{proof}
(1)  Let $a,b\in B$. Clearly the semi-brace identity yields
$ab =1_{\circ} \circ (ab) = (1_{\circ} \circ a) (1_{\circ} \circ (1_{\circ} b)) = a1_{\circ}b$.

(2) First, we show that the ideal $B^{1}1_{\circ}B^{1}$ is multiplicatively 
closed for the operation $\circ$. Let $a,b,c,d \in B^{1}$. 
If both $c$ and $d$ are the multiplicative  identity of $(B^{1},\cdot )$, 
then 
   $$ (a1_{\circ} b) \circ (c1_{\circ} d)  = (a1_{\circ} b) \circ 1_{\circ} = 
        a1_{\circ}b \in B^1 1_{\circ} B^1.$$
Suppose next that  $c$ or $d$ is an element of $B$. 
As both cases are dealt with  analogously,  assume that $c \in B$. 
Then, 
     $$ (a1_{\circ}b) \circ (c1_{\circ}d) = 
         ((a1_{\circ}b) \circ c) \, ( (a1_{\circ}b) \circ ( \overline{a1_{\circ}b} 1_{\circ}d)) 
         \in B^2.$$ 
Because of  Lemma~\ref{S2isS1S}(1), it follows that $(a1_{\circ} b) \circ 
(c1_{\circ} d) 
\in B^11_{\circ} B^1$, as desired.
	
Second, we show that $B^1 1_{\circ} B^1$ is closed for taking inverses in 
$(B,\circ )$. 
Let $a,b \in B^1$. 
Clearly, $1_{\circ}  \in B^1 1_{\circ} B^1$. 
We need to show that if $a,b\in B^{1}$, not both in $B^{1}\setminus B$, then $\overline{a 1_{\circ} b} \in B^{1}1_{\circ} B^{1}$.
We assume $a\in B$, the other case is proved analogously.
Since, 
   $$ 1_{\circ}  = \overline{a1_{\circ} b} \circ (a1_{\circ} b) 
       = (\overline{a1_{\circ} b} \circ a) ( \overline{a1_{\circ}b}  \circ 
          ( a1_{\circ} b 1_{\circ} b)) \in B^2.$$
We obtain that there exist $z,w \in B$ such that $zw=1_{\circ}$. Hence,
    $$ \overline{a1_{\circ} b} = \overline{a1_{\circ}b} \circ (zw) 
       = (\overline{a1_{\circ} b}\circ z) (\overline{a1_{\circ} b} \circ ( a1_{\circ} b w)) \in
       B^{2}=B^{1}1_{\circ} B^{1}.$$
The last equality makes use of part (1).

\end{proof}

\begin{lemma}\label{subgroepcomplem}
Let $(B,\cdot, \circ)$ be a left semi-brace.
Then
$(B \setminus B^11_{\circ}B^1) \cup \lbrace 1_{\circ} \rbrace$ is a subgroup of $(B,\circ )$.
\end{lemma}

\begin{proof}
Put $D = (B \setminus B^11_{\circ} B^1) \cup \lbrace 1_{\circ} \rbrace$. 
We need to show that $\overline{a} \circ b \in D$ for distinct elements $a$ and $b$ in $D$.
Suppose the contrary, that is, assume there exist distinct $a,b\in D$ such that 
$ \overline{a} \circ b = s 1_{\circ} t$.
Note that not both $s$ and $t$ can be in $B^1\setminus B$, as $a\neq b$. Hence, 
using the semi-brace identity, we get that
$b = a \circ (s1_{\circ}t) \in B^2$. 
So, by Lemma~\ref{S2isS1S}, $b\in  B^1 1_{\circ}  B^1$,  a contradiction.

\end{proof}

\begin{lemma} \label{basics}
	Let $(B,\cdot, \circ)$ be a left semi-brace.
The following properties hold.
\begin{enumerate}
\item  $ B = B^1 1_{\circ} B^1$.
\item  $B^1_{\circ} 1_{\circ}$ is a subgroup of $(B,\circ)$ and $1_{\circ}$ is an idempotent of $(B,\cdot)$.
In particular, $B^{1}1_{\circ}= B1_{\circ}$, $1_{\circ} B=1_{\circ} B^{1}$, $B=B^{1}1_{\circ} B^{1}=B1_{\circ}B$ and thus $B1_{\circ}$ is a sub semi-brace of $B$.
\item $(1_{\circ}B,\circ)$ is a semigroup. In particular,  if $B$ is finite then $(1_{\circ}B,\circ)$ is a subgroup of $(B,\circ )$ and thus $1_{\circ}B$ is a subsemi-brace of $B$. 
\end{enumerate}
\end{lemma} 

\begin{proof}
From Lemma~\ref{S2isS1S} and Lemma~\ref{subgroepcomplem} we know that the group $(B,\circ)$ is the union of the two subgroups $B^{1}1_{\circ}B^{1}$ and $(B\setminus B^{1}1_{\circ}B^{1}) \cup \{ 1_{\circ} \}$.
Because the intersection of these two subgroups is $\{ 1_{\circ}\}$ it follows that $B=B^{1}1_{\circ}B^{1}$ or
$B^{1}1_{\circ}B^{1}=\{ 1_{\circ} \}$.
In the second case, $b1_{\circ} =1_{\circ} =1_{\circ} b$ for all $b\in B$, in 
contradiction with Lemma~\ref{nozero} if $|B| \geq 2$.
This proves the first part.

To prove the second part, put $K=B^{1} 1_{\circ}$.
Clearly, $1_{\circ} \in K$. 
Let $x \in B$. Then, by Lemma~\ref{S2isS1S},
    $$ 1_{\circ}  = \overline{x1_{\circ}} \circ x1_{\circ} =( \overline{x1_{\circ}} \circ x) ( \overline{x1_{\circ}} \circ (x1_{\circ}1_{\circ})) = (\overline{x1_{\circ}} \circ x) 1_{\circ}.$$
Thus, there exists $z \in B$ such that $z1_{\circ} = 1_{\circ}$. 
Thus, again by Lemma~\ref{S2isS1S},
    $$\overline{x1_{\circ}} = \overline{x1_{\circ}} \circ (z1_{\circ}) = (\overline{x1_{\circ}} \circ z) (\overline{x1_{\circ}} \circ (x1_{\circ}1_{\circ})) = (\overline{x1_{\circ}} \circ z) 1_{\circ} \in B1_{\circ}.$$
Moreover, by Lemma~\ref{S2isS1S}(1),
          $$ 1_{\circ} = z1_{\circ} = z1_{\circ}1_{\circ} = 
          1_{\circ}1_{\circ}.$$
By the previous paragraph, for $x,y\in B$,
    \begin{eqnarray*}
    (x1_{\circ}) \circ (y1_{\circ}) &=&( (x1_{\circ}) \circ y) (x1_{\circ} \circ \overline{x1_{\circ}}1_{\circ})
     \, = \, ((x1_{\circ}) \circ y) (x1_{\circ} \circ (\overline{x1_{\circ}}))\\
     & =& ((x1_{\circ}) \circ y) 1_{\circ} \in B1_{\circ}.
     \end{eqnarray*}
Hence we have shown that indeed  $(B1_{\circ}, \circ)$ is a subgroup of  
$(B,\circ )$.

For the third part it is sufficient to note that, for $x,y\in B$, 
	$$(1_{\circ}x) \circ (1_{\circ}y) = (1_{\circ}x) (1_{\circ}x \circ (\overline{1_{\circ}x} 1_{\circ}y)) \in 1_{\circ}B.$$
\end{proof}

\begin{lemma}\label{1s1subgroup}
Let $(B,\cdot,\circ)$ be a left semi-brace such that $1_{\circ}B$ is a subgroup of $(B,\circ)$. 
Then, $1_{\circ}B1_{\circ}$ is a subgroup of the semigroup $(B,\cdot)$. 
In particular, if $B$ is finite, then $1_{\circ}B1_{\circ}$ is a subgroup of 
$(B,\cdot)$.
\end{lemma}

\begin{proof}
From Lemma~\ref{basics},  $1_{\circ}$ is the  identity of the monoid $(1_{\circ}B1_{\circ},\cdot )$. 
To prove it is a group, let $1_{\circ}x1_{\circ} \in 1_{\circ}B1_{\circ}$. 
By the assumption, there exists  $y \in B$ such that 
    $$ 1_{\circ} = (1_{\circ}x) \circ (1_{\circ} y) = 1_{\circ} x (1_{\circ}x \circ (1_{\circ}yy)).$$ 
Hence, there exists $z \in B$ such that $ 1_{\circ}xz = 1_{\circ}$. 
So, because of Lemma \ref{S2isS1S},  
    $$ 1_{\circ} = 1_{\circ}xz= 1_{\circ}x1_{\circ}1_{\circ}z.$$ 
As $1_{\circ}$ is an idempotent, it follows that 
    $$ 1_{\circ} = 1_{\circ}1_{\circ} = 1_{\circ}x1_{\circ}1z1_{\circ}.$$ 
Thus every element of $1_{\circ} B 1_{\circ}$ has a right inverse in the monoid $(1_{\circ}B1_{\circ},\cdot)$, as desired.
\end{proof}

\begin{theorem} \label{CompSimp}
Let $(B,\cdot,\circ)$ be a left semi-brace. 
If $1_{\circ}B$ is a subgroup of $(B,\circ)$ (for example if $B$ is finite), 
then $(B,\cdot)$ is a completely simple semigroup with maximal subgroup $1_{\circ}B1_{\circ}$.
\end{theorem}

\begin{proof}
From Lemma~\ref{basics} we know that $B=B1_{\circ}B$. Let $b \in B$.
Hence, by the assumption and Lemma~\ref{1s1subgroup}, $B=B1_{\circ}B = B1_{\circ} b 1_{\circ} B$.
So, by Lemma~\ref{S2isS1S}, $B=BbB$ for every $b\in B$.
Therefore, every principal ideal, and thus every ideal,  of $B$ is trivial, i.e. $(B,\cdot )$
is a simple semigroup.
Since $1_{\circ} B 1_{\circ}$ is a subgroup of $(B,\cdot )$ by Lemma~\ref{1s1subgroup}, the idempotent $1_{\circ}$ is a
primitive idempotent of $(B,\cdot)$. Hence, $(B,\cdot )$ is a completely simple 
semigroup.
\end{proof}

Because of the earlier mentioned result on completely simple semigroups, we get from the theorem (under the same assumptions)
that the multiplicative semigroup of the left semi-brace $(B,\cdot )$ is  a matrix semigroup of the type
$\mathcal{M}(G, I,J,P)=\{ (g,i,j)\mid g\in G,\; i\in I, \; j\in J\}$. 
Let us fix an index in both $I$ and $J$. For simplicity we denote both as $1$. 
Also denote $1$ for the identity in $G$. 
It  also is  well known that (see \cite[Theorem 3.4.2]{Howie}), without loss of generality, one may assume that $p_{1j}=p_{i,1}=1$ for all $i\in I$ and $j\in J$
and, as $1_{\circ}$ is an idempotent, $1_{\circ}=(1,1,1)$.

\begin{corollary} \label{rowcol}
Let $(B,\cdot,\circ)$ be a completely simple left semi-brace. Then, as multiplicative semigroups, 
$(B,\cdot)\cong \mathcal{M}(G, I,J,\mathcal{I}_{|J|,|I|})=\{ (g,i,j)\mid g\in 
G,\; i\in I, \; j\in J\}$  with $G=1_{\circ} B 1_{\circ}$.
Furthermore,
\begin{enumerate}
\item  the first row $R=\left\lbrace (g,1,j) \mid g \in G, j \in J \right\rbrace$ is a left cancellative left subsemi-brace, and 
\item the first column $K= \left\lbrace (g,k,1) \mid g \in G, k \in I \right\rbrace$ is a right cancellative left subsemi-brace.
\item $(G,1,1)$ is a left subsemi-brace. Moreover, 
$(G,\cdot)$ is a group.
\end{enumerate}
If, furthermore, $G=\{ 1 \}$ then both $R$ and $K$ are two-sided semi-braces.
\end{corollary}

\begin{proof}
Because of Theorem~\ref{CompSimp} and the above remarks, we may assume that, as 
a multiplicative semigroup,
$B=\mathcal{M}(G, I,J,P)$ with $p_{i1}=p_{1j}=1$ and $1_{\circ} =(1,1,1)$.
Let $(g,i,j), \, (h,k,l) \in B$. Because of  Lemma~\ref{S2isS1S},
$(g,i,j) (1,1,1)(h,k,l)=(g,i,j)(h,k,l)$ and thus $gp_{j1}p_{1k}h=gp_{jk}h$. 
Hence, $p_{jk}=p_{j1}p_{1k}=1$, for all $j\in J$ and $k\in I$.
This proves the first part of the statement.

Because of Lemma~\ref{basics}, we know that $B1_{\circ}$ is a left 
subsemi-brace. Clearly $B1_{\circ}=K$ and $K$ is a right cancellative 
multiplicative semigroup.
Note that $R=1_{\circ}B$. Hence, from Lemma~\ref{basics}, we know that 
$(R,\circ )$ is a submonoid of $(B,\circ )$. To prove part $(1)$ it remains to 
show that $(R,\circ)$ is closed under taking inverses. So let $(g,1,j)\in R$. 
Then,
\begin{eqnarray*}
(1,1,1) &=&\overline{(g,1,j)} \circ (g,1,j)
                        = \overline{(g,1,j)} \circ  \left( (1,1,1)\, (g,1,j) \right)\\
 & =& (\overline{(g,1,j)} \circ (1,1,1) )\, \left( \overline{(g,1,j)} \circ ((g,1,j)(g,1,j)) \right)\\
 & =& \overline{(g,1,j)}  \, \left(\overline{(g,1,j)} \circ ((g,1,j)(g,1,j)\right).
 \end{eqnarray*}
Hence, it follows that $\overline{(g,1,j)}\in R$, as desired.

That  $(G,1,1)$ is a left subsemi-brace follows at once from the fact that it 
is the intersection of the subsemi-braces  $R$ and $K$. By the definition of 
$\mathcal{M}(G,I,J,\mathcal{I}_{|J|,|I|})$, $(G,\cdot)$ is a subgroup of 
$(B,\cdot)$.

Finally, assume $G$ is trivial. We prove that $K$ also is a right semi-brace. Note that $K$ is a left zero semigroup, i.e. $ab=a$ for all $a,b\in K$.
Let $a,b,c \in K$.  Then, $ (bc) \circ a = b \circ a = (b\overline{a}) \circ a =\left( (b\overline{a}) \circ a \right) (c\circ a)$. 
So, $K$ is a right semi-brace. This shows that $K$ is a two-sided semi-brace.

Second, we prove that $R$ is a two-sided 
semi-brace. Recall that if $G$ is trivial, $(R,\cdot,\circ)$ is a right zero 
left semi-brace. Then, for $a,b,c \in R$, $ (b c) \circ a = c \circ a$. 
Moreover, $ \left( \left( b\overline{a} \right) \circ a \right) 
\left( c \circ a \right) = c \circ a$. It follows that indeed $(R,\cdot,\circ)$ is a two-sided semi-brace.
\end{proof}

 Throughout we will use the notation used in Corollary 2.9 for a completely 
 simple left semi-brace.
So,  an element $b\in B$ can be written as $(g,i,j)$ and we call $g=(g,1,1)\in  
G=1_{\circ} B 1_{\circ}$ the group component of $B$.
Note that from the proof of the last part one has proven the following.

\begin{example}
Let $(B,\cdot )$ be a right (repectively left) zero semi-group. If $(B,\circ)$ 
is a group, then $(B,\cdot , \circ)$ is a two-sided semi-brace.
\end{example}

It has been shown in Corollary~\ref{rowcol}   that if $B=\mathcal{M}(G, 
I,J,\mathcal{I}_{|J|,|I|})$ is a completely simple semi-brace then $G$ is a 
skew left brace.
Conversely, if $(G,\cdot , \circ_1)$ is a skew left brace  and  also 
$(I,\circ_2 )$ and $(J,\circ_3 )$ are groups   then the multiplicative semigroup 
$\mathcal{M}(G,I,J,\mathcal{I}_{|J|,|I|})$ 
becomes a left semi-brace for the group operation $\circ$ defined as follows
  \begin{eqnarray}
  (g,i,j) \circ (h,k,l) &=& (g\circ_1 h, i\circ_2 k, j\circ_3 l). \label{productcompetelysimple}
  \end{eqnarray}
However, not every completely simple semi-brace is of this type as shown by the following example.
The aim  of this section 
is to determine the structure of completely simple
left semi-braces.

\begin{example}\label{special}
Consider the set $B=\mathbb{Z}_3\times \mathbb{Z}_2$ This is a semigroup for the operation 
$(x,y)(z,w) = (x,y+w)$, where $+$ is the sum in $\mathbb{Z}_2$. Consider the 
bijection $B \longrightarrow C_6 = 
\left<\xi\mid\xi^6 = 1 \right>$ defined by $(i,j) \mapsto \xi^{i + 3j}$ for $i 
\in \left\lbrace 0,1,2 \right\rbrace$ and $j \in \left\lbrace 0,1 
\right\rbrace$. Hence,  the  group structure of $C_6$ induces a group structure 
on $B$.
Then $B$ is a right cancellative left semi-brace such that the idempotents of 
$(B,\cdot)$ do not form a subsemi-brace of $B$. In particular, this is not an 
example 
of the previous type. 
\end{example}

\begin{proof}
	Obviously, the semigroup $(B,\cdot)$ is right cancellative. Identify the elements of $B$ with their images under the bijection $(i,j) \mapsto \xi^{i+3j}$.
	We will check the equation $$ \xi^i \circ ( \xi^k \xi^j ) = \xi^i \circ \xi^k \xi^i \circ ( \xi^{6-i} \xi^j).$$
	Clearly, for any $0 \leq t <6$, $\xi^t$ and $\xi^{t+3}$ have the same first 
	component. Thus, $\xi^t \xi^3 = \xi^{t+3}.$ Notice that the elements with 
	$0$ in their $\mathbb{Z}_2$-component are precisely the idempotents of 
	$(B,\cdot)$. Clearly, these idempotents are right identities.
	To prove the semi-brace identity, we consider two mutually exclusive cases:
	
	First, suppose $\xi^j$ is an idempotent. Then, $$ \xi^i \circ ( \xi^k 
	\xi^j) = \xi^i \circ \xi^k = \xi^{k+i}.$$
	On the other hand, $$ \left(\xi^i \circ \xi^k\right) \left( \xi^i \circ ( 
	\xi^{6-i} \xi^j)\right) = \xi^{k+i} \left( \xi^i \circ \xi^{6-i}\right) = 
	\xi^{k+i}.$$
	
	Second, suppose now that $\xi^j$ is not an idempotent. Clearly, the 
	equation $ \xi^i \circ (\xi^k \xi^j) = \left( \xi^i \circ \xi^k \right) 
	\left( \xi^i \circ \left( \xi^{6-i} \xi^j \right)\right)$ holds if $i=0$. 
	We will now show the equation also holds for $i=1$, the other values are 
	analogous.  As $\xi^j$ is not an idempotent, $\xi^j = (\alpha, 1)$, for 
	some $\alpha \in \left\lbrace 0, 1, 2 \right\rbrace$. Thus, if $\xi^k = 
	(a,b)$, it follows that $\xi^k \xi^j = (a, b+1)$. Calculating the left 
	hand-side,
	$$ \xi \circ (\xi^k \xi^j) = \xi \circ \xi^{k+3} = \xi^{k+4}.$$
	Similarly, calculating the right hand-side and using the remark above,
	$$ \xi \circ \xi^k \xi \circ (\xi^5 \xi^j) = \xi^{k+1} \xi \circ \xi^2 = \xi^{k+1} \xi^3 = \xi^{k+4}.$$
	Hence, the equality holds.
	
	Clearly, the set of idempotents of $B$ is identified in $C_6$ (under the same identification as before) with $\left\lbrace 1, \xi, \xi^2 \right\rbrace$, which is not a subgroup of $C_6$. Hence, the set of idempotents of $B$ is not a subsemi-brace.
\end{proof}

We finish this section with showing that skew left braces $(B,\cdot , \circ )$ are precisely 
the left semi-braces $(B,\cdot ,\circ)$ with $(B,\cdot )$ a group.
In order to do so it is convenient to introduce the $\lambda$ map associated with a left 
semi-brace $(B,\cdot ,\circ )$.
In case of (skew) braces these maps turned out the be very useful (see for 
example \cite{bachiller2015classification,cedo2014braces,guarnieri2017skew}). 
For $a,b\in B$, note that $a\circ 1_{\circ} b = a(a\circ (\overline{a} b)) $. Hence, one defines 
                       $$\lambda_{a} : B \longrightarrow B: b\mapsto a\circ (\overline{a} b).$$
So, \begin{eqnarray}\label{lambdaproduct} a\circ 1_{\circ} b &=& 
a\lambda_{a}(b) \end{eqnarray} and thus the $\lambda$-maps link the 
multiplicative structure $\cdot$ with the
group structure $\circ$. 

The $\rho$-map associated to the semibrace $B$ is defined as follows
	$$ \rho_a :  B \longrightarrow B : b \mapsto \overline{\left(\overline{b} a \right)} \circ a.
	  $$

\begin{lemma}\label{lambdaidempot}
	Let $\left(B, \cdot, \circ \right)$ be a left semi-brace. For any $a \in 
	B$, $\lambda_a \in \textup{End}(B,\cdot)$. Moreover, $ \lambda: (B,\circ) 
	\longrightarrow \textup{End}(B,\cdot): x \mapsto \lambda_x$ is a 
	homomorphism. Furthermore, for any $a \in B$, we have that $\lambda_a(E(B)) 
	\subseteq E(1_{\circ}B)$.
\end{lemma}

\begin{proof}
Let $a,b,x,y \in B$. 
By the semi-brace property, 
   \begin{align*} \lambda_a(x y)  &= a \circ \left( \overline{a}xy \right)  =  \left(a \circ (\overline{a}x) \right) (a \circ ( \overline{a} y)) = \lambda_a(x) \lambda_a(y) . \end{align*}
and
	\begin{eqnarray*}
	  \lambda_{a \circ b}(x) 
	           &=& \left( a \circ b \right) \circ \left( (\overline{a \circ b}) x \right) 
	           = a \circ \left( b \circ \left( (\overline{b} \circ \overline{a}) x \right) \right) \\
	           &= &a \circ \left( (b \circ \overline{b} \circ \overline{a}) \left( b \circ \left( \overline{b} x\right) \right) \right) 
	               = a \circ \left( \overline{a} \lambda_b(x) \right)\\
	             &  = & \lambda_a \lambda_b(x).
	  \end{eqnarray*}
\end{proof}

The $\rho$-maps behave different in general.

\begin{lemma} \label{rhomap1}
	Let $(B,\cdot,\circ)$ be a left semi-brace. Then, for any $x,y \in B$, we 
	have that $\lambda_x(y) \in 1_{\circ}B$. If $(B,\cdot,\circ)$ is a completely 
	simple semi-group, then $\rho_{x}(y) \in B1_{\circ}$.
\end{lemma}
\begin{proof}
	Using the definition, it follows that $$ \lambda_x(y) = x \circ \left( 
	\overline{x}y\right) = \left(x \circ \overline{x}\right) \left( x 
	\circ\left( \overline{ x}y\right)\right) = 1_{\circ}\lambda_x(y).$$
	Let $x=(g,s,r)$ $$ \rho_{x}(y) = \overline{ (\overline{y}x)}\circ 
	(x(1,1,r)) = \left(\overline{(\overline{y}x)} \circ x\right) 
	\overline{(\overline{y}x)} \circ (\overline{y}x(1,1,r)) = \rho_{x}(y)1_{\circ}.$$
\end{proof}


\begin{prop}\label{charactanti}
	Let $(B,\cdot, \circ)$ be a left semi-brace. The following properties are equivalent.
\begin{enumerate}
\item[(1)] $\rho: \left( B, \circ \right) \longrightarrow \textup{Map}(B,B)$ is 
an anti-homomorphism.
\item[(2)] $c \left(a\circ \left(1_{\circ}b\right)\right) = c \left( a \circ b 
\right)$ for all $a,b,c \in B$.
\item[(3)] $(B,\cdot)$ is completely simple and, for any $(g,i,j) \in B$ and 
$(1,k,l) \in E(B)$, if $(h,r,s) = (g,i,j) \circ (1,k,l)$, then 
$h=g$.
\end{enumerate}
Moreover, in these cases, 
the  idempotents $E(B)$  form a left subsemi-brace as well as the idempotents $E(B1_{\circ})$ of the left
subsemi-brace $B1_{\circ}$.
\end{prop}

\begin{proof}
Let $a, b, x \in B$.
The left semi-brace property yields that 
\begin{eqnarray*}
	\rho_b \rho_a(x) &=& \overline{\left( \overline{\rho_a(x)}b\right)} \circ b \; = \; \overline{\left( \overline{\rho_a(x)}b\right)} \circ \overline{a} \circ a \circ b \\
	 &=& \overline{a \circ \left( \overline{\rho_a(x)}b\right)} \circ (a \circ b) 
	 \; =\;   \overline{a \circ \left( \left( \overline{a} \circ 
	 \left(\overline{x}a\right) \right)b \right)} \circ \left(a \circ b \right) 
	 \\ 
	 &=& \overline{ \overline{x}a \left( a \circ \left( \overline{a} b \right)\right)} \circ \left(a\circ b \right) 
	 \;= \; \overline{ \overline{x} \left(a \circ \left(1_{\circ}b \right)\right)} \circ \left( a \circ b \right) \\  
\end{eqnarray*}
and
\begin{eqnarray*}
	\rho_{a\circ b}(x) &=&  \overline{ \overline{x} \left(  a \circ b \right)} \circ \left( a \circ b \right) 
	\end{eqnarray*}
Hence, $\rho$ is  an anti-homomorphism precisely when $\overline{x} (a\circ (1_{\circ} b)) = \overline{x} (a\circ b)$. That is, when property (2) holds.

Assume (2) holds.
We now prove the last claim. From Lemma~\ref{basics} we know that $1_{\circ}$ is a multiplicative idempotent. Hence, for any $b\in B$, property (2) yields that
 $$1_{\circ} = 1_{\circ} (1_{\circ} b \circ \overline{1_{\circ} b})= 1_{\circ} (1_{\circ} b \circ (1_{\circ} \overline{1_{\circ} b}) = 1_{\circ} b\lambda_{1_{\circ}b}(\overline{1_{\circ}b})$$
and thus $1_{\circ} \in 1_{\circ} bB$. Therefore, by Lemma \ref{basics}, 
 $$B=B1_{\circ} B = B1_{\circ} bB \subseteq BbB.$$
So, $(B,\cdot)$ is a simple semigroup and every element of $1_{\circ}B1_{\circ}$ has a right inverse. Hence $1_{\circ}B1_{\circ}$ is a multiplicative subgroup of
$(B,\cdot)$.
Thus, $(B,\cdot )$ is completely simple  and, from Corollary~\ref{rowcol}, we 
know that $B=\mathcal{M}(G,I,J,\mathcal{I}_{|J|,|I|})$ and  $B1_{\circ} = K =\{ 
(g,i,1) \mid g\in G,\; i\in I\}$ is a left subsemi-brace. 

Put $E=E(B)$. Let $c,b \in E$. 
Because of $(2)$, for any $a \in B$ we have $$ c(a \circ (1_{\circ}b)) = c(a 
\circ b).$$ By the semi-brace property and Lemma~\ref{lambdaidempot}, $$ 
\underbrace{c}_{\in E} a \underbrace{a \circ (\overline{a}b)}_{\in E} = 
\underbrace{c}_{\in E} (a\circ b).$$ It follows  that the group component of $a$ 
is precisely the group component of $a \circ b$. Equivalently, for any 
$(h,k,l), (1,s,r)\in B$, 
         $$(h,k,l)\circ (1,s,r)=(h,x,y).$$
 Since the idempotents are precisely the elements $(1,k,l)$ it then is clear 
 that they form a subsemi-brace of $B$.
 Hence we have shown that (3) holds.

We now prove that (3) implies (2). Hence $B=\mathcal{M}(G, I,J, \mathcal{I}_{|J|,|I|})$.
Let $c, a=(g,i,j), b=(h,k,l) \in B$. Then, by the semi-brace property and Lemma~\ref{lambdaidempot},
 $$ c ( a 
\circ (1_{\circ}(h,k,l))) = c a  \lambda_a((1,k,l)) \lambda_a(h,1,1) \lambda_a(1,1,l)$$
and
 $$ c (a \circ b) = c (a \circ 
(1,k,1)) \lambda_a(h,1,1) \lambda_{a}(1,1,l).$$ 
By assumption, $a\circ (1,k,1) = (g,s,t)$ for some $s\in I$ and $t\in J$. So, it follows that
 $$ c(a \circ b) = c( a \circ (1_{\circ}b)).$$
This proves (2).
 
 It remains to show that under the assumption of (1) that $E(K)$ is a left subsemi-brace. 
 Applying the assumption (3)  for $h=1$ and $l=r=1$, we obtain that the group component of $(1,k,1)\circ (1,s,1)$ is $1$. Hence, since $K$ is a left subsemi-brace, 
$(1,k,1)\circ (1,s,1)  \in K$. 
This shows that the idempotents of $(K,\cdot)$ form a multiplicative 
subsemigroup of $(B,\circ)$.
Also, with  $(h,k,l) =\overline{(1,s,1)}$ and $r=1$ we get that
$(1,1,1)= \overline{(1,s,1)} \circ (1,s,1)$. Thus, if $(h,k,l) = 
\overline{(1,s,1)}$, it follows that $h=1$. Furthermore, $$ \overline{(1,s,1)} 
= \overline{(1,s,1)} \circ (1,1,1) = \overline{(1,s,1)} \left( 
\overline{(1,s,1)} \circ (1,s,1)\right) = \overline{(1,s,1)}1_{\circ}.$$
Hence, 
$\overline{(1,s,1)}\in K$ anf thus  $E(K)$ is a left subsemi-brace 
of $B$.
\end{proof}

Some examples of left semi-braces with $\rho$ an anti-homomorphism:
(i)  finite completely simple left semi-braces $B$ with $|1_{\circ} B 
1_{\circ}|=1$, i.e. $(B,\cdot)$ consists  of idempotents,\\
(ii)  the semigroups  $(B,\cdot) = 
	\mathcal{M}(G,|I|,|J|,\mathcal{I}_{|J|,|I|})$ and with $\circ$ defined by 
	equation (\ref{productcompetelysimple}) and\\
	(iii) left cancellative left semi-brace. Hence we recover a result of  
Catino, Colazzo and Stefanelli  \cite{catino2017semi}. 

It can be seen that for the left semi-brace defined in Example \ref{special} 
$\rho$ is not an anti-homomorphism.

\begin{corollary} \label{GroupIdempotent}
Let $(B,\cdot,\circ)$ be a left semi-brace. If  $\rho$ is an anti-homomorphism 
then for every $b\in B1_{\circ}$ there exist unique $g\in 1_{\circ}B1_{\circ} $ 
and an idempotent  $e\in E(B1_{\circ})$
such that $b=g\circ e$ (and thus $g$ is the group component of $b$). In 
particular $B1_{\circ}=1_{\circ}B1_{\circ}  \circ E(B1_{\circ})$. Moreover, 
since $E(B1_{\circ})$ is a subsemi-brace, $b\circ f$ and $b$ have the same 
group component for any $f \in E(B1_{\circ})$.
\end{corollary}

\begin{proof}
Assume that $\rho$ is an anti-homomorphism. Thus, by Proposition~\ref{charactanti}, $B=\mathcal{M}(G,I,J, \mathcal{I}_{|J|,|I|})$ with $G=1_{\circ} B 1_{\circ}$.
Let $b\in B1_{\circ}$. Then, $b=kg$ for some $g\in G$ and $k\in E(B1_{\circ})$. Because $B1_{\circ}$ is a left subsemi-brace, by Corollary~\ref{rowcol}, we may write $\overline{b}= fh$ with $f\in E(B)$ and $h\in G$. 
Let $h^{-1}$ denote the inverse of $h$ in $(G,\cdot )$. Then,
\begin{eqnarray*}
b \circ f &=& (kg) \circ (fhh^{-1}) = ((kg) \circ (fh)) \lambda_{kg}(h^{-1} )= 1_{\circ}   \lambda_{kg}(h^{-1} ).
\end{eqnarray*}
Because  $ \lambda_{kg}(h^{-1} )\in B1_{\circ}$ we get that $b \circ f \in G$. Furthermore, by Proposition~\ref{charactanti}, $b\circ f$ has group component equal to $g$.
Hence, $b\circ f=g$. So, $b=g\circ \overline{f}$. As $f\in E(B1_{\circ})$, the result follows.
\end{proof}

\begin{lemma}\label{behaviorlambda}
	Let $(B,\cdot,\circ)$ be a completely simple left semi-brace. Then, $\lambda_e(x) = e \circ x$, for any 
	$x \in 1_{\circ}B$ and $e \in E(1_{\circ} B)$.
\end{lemma}
\begin{proof}
	Let $e \in E(1_{\circ}B)$ and $x \in 1_{\circ}B$. Then, $\lambda_e(x) = e \circ ( ex) = e 
	\circ x$.
\end{proof}

\begin{definition} \cite{guarnieri2017skew}
	Let $(B,\cdot)$ and $(B,\circ)$ be groups. The inverse of $a$ in $(B,\cdot )$ is denoted by $a^{-1}$. If, for all $a,b,c$ one has 
	$$ a \circ (bc) = \left( a \circ b\right) a^{-1}\left( a \circ 
	c\right),$$
	then $(B,\cdot,\circ)$ is called a skew left brace. 
\end{definition}

\begin{prop} Let $B$ be a set and suppose $(B,\cdot )$ and $(B,\circ )$ are groups. Then $(B,\cdot,\circ)$ is
 a left semi-brace if and only if  $(B,\cdot,\circ)$ is a skew left brace.
\end{prop}
\begin{proof}
Assume $(B,\cdot,\circ)$ is a left semi-brace.
	Let $0 \in B$ denote the identity element in $(B,\cdot)$. Then, $$0=1_{\circ} \circ 
	0 = 1_{\circ} \circ \left( 00 \right) = (1_{\circ}\circ 0) (1_{\circ}\circ (1_{\circ}0)) = 0 1_{\circ}0 = 1_{\circ}.$$
	Hence, $1_{\circ} = 0$. For any $x \in B$, it follows that $$ \lambda_{1_{\circ}}(x) = 1_{\circ}
	\circ 1_{\circ}x = 1_{\circ}x = x.$$
	Because of Lemma~\ref{lambdaidempot} it this follows, for any $a,b \in B$,  
	that
	 $$
	b = \lambda_{1_{\circ}}(b) = \lambda_{\overline{a}} \lambda_a (b) = \overline{a} 
	\circ 
	a \lambda_a(b).$$
	Hence, using that both $(B,\cdot)$ and $(B,\circ)$ are groups, it follows 
	that $$ a^{-1} ( a \circ b) = \lambda_{a}(b).$$
	This shows that $(B,\cdot,\circ)$ is a skew left brace.
	
The converse is left to the reader.
\end{proof}

\section{Decomposition into matched products}

{\rm In this section we will show that many left semi-braces can be decomposed as a 
matched product as defined by Catino, Colazzo and Stefanelli 
\cite{catino2017semi} (which on its turn is an extension of the definition of 
matched products of left braces as defined in 
\cite{bachiller2015extensions}).  
The definition 
given in \cite{catino2017semi} is under the  
assumption 
that the left semi-braces involved are left cancellative. However, this is not  needed to show that the proposed operations indeed define a left semi-brace.
}

\begin{definition}\label{Matchedproductsemibrace}
	Let $(B,\cdot, \circ)$ and $(S,\cdot,\circ)$ be left semi-braces. Let 
	$\delta: B \longrightarrow \textup{Aut}(S,\cdot)$ be a right action of the 
	group 
	$(B,\circ)$ on the set $S$ and $\sigma: S \longrightarrow \textup{Sym}(B)$ 
	 a left action of the group $(S, \circ)$ on the set $B$.  
	Assume the 
	following properties hold for any $x,y \in S$ and $a,b \in B$,
	\begin{enumerate}
		\item $\sigma_x(a \circ b) = \sigma_x(a) \circ \sigma_{ \delta_a(x)}(b),$
		\item $\sigma_x(1_{\circ}) = 1_{\circ},$
		\item $\delta_a(x \circ y) = \delta_{\sigma_{y}(a)}(x) \circ \delta_a(y),$
		\item $\delta_a(1_{\circ}) = 1_{\circ},$
		\item $ \overline{ \delta_a ( \overline{ xy}) } = \overline{ \delta_a( \overline{x})} \quad \overline{ \delta_a( \overline{y})}$.
	\end{enumerate}
Then the following operations define a left semi-brace structure on    on $B 
\times S$  
\begin{align*} (a,x) (b,y) &= (ab, xy), \\ ( a,x) \circ (b,y) &= \left( a \circ 
\sigma_{ \overline{ \delta_{a}( \overline{x})}}(b), x \circ \overline{ 
\delta_{\overline{ \sigma_{ \overline{x}}(a)}}(\overline{y})} \right) .
	\end{align*}
This is called the matched product of the left semi-brace $B$ and $S$ by $\delta$ and $\sigma$, denoted $B \bowtie S$. 
\end{definition}

 That $B \bowtie S$ actually is a left semi-brace can be proven exactly in 
the same way as in the thesis of Colazzo \cite[Theorem 3.1.1]{Colazzothesis}.
Furthermore, it is shown that if, moreover, both $B$ and $S$ are left 
cancellative then so is the matched product.

We are now in a position to prove the main result of this section, generalizing Theorem~\ref{Colazzochar}.

\begin{theorem}\label{characttheorem}
	Let $(B,\cdot,\circ)$ be a completely simple left semi-brace such that 
	$\rho$ is anti-homomorphism. Then 
	\begin{enumerate}
	\item $R=E(1_{\circ}B)$ is a right zero left
	semi-brace $(R, \cdot, \circ)$, 
	\item $K=B1_{\circ}$ is a right cancellative left 	semi-brace $(K, \cdot, \circ)$,
	\item  $(B, \cdot, \circ) \cong (K, \cdot,  \circ) \bowtie (R, \cdot, \circ)$. 
	\end{enumerate}
Furthermore, as multiplicative semigroups 
 $$B = \mathcal{M}(1_{\circ}B1_{\circ}, |E(K)|, 	|E(R)|, \mathcal{I}_{|E(R)|,|E(K)| }).$$
\end{theorem}
\begin{proof}

Define for $(g,s,1) \in K$ the map $ \delta_{(g,s,1)}: R \longrightarrow R$ 
with $\delta_{(g,s,1)} \left( (1,1,l) \right) = \overline{ \lambda_{ \overline{ 
(g,s,1)}} \left( \overline{ (1,1,l)} \right)}$.

First, we show that $\delta_{(1,1,1)} = \textup{id}_R$. Let $(1,1,l) \in R$, 
then \begin{align*} \delta_{(1,1,1)} \left( (1,1,l) \right) &= \overline{ 
\lambda_{(1,1,1)} \left( \overline{ (1,1,l)} \right)} = \overline{ 
\left(\overline{ (1,1,1)}\overline{ (1,1,l)} \right)} = (1,1,l).
\end{align*}
Thus, $\delta_{(1,1,1)} = \textup{id}_R$.

Second, we show that $\delta$ induces an action $\delta: K \longrightarrow 
\textup{Sym}(R)$. To so, let $(g,s,1),(h,t,1) \in K$ and $(1,1,l) \in R$. Then,
\begin{eqnarray*}
\delta_{(g,s,1)} \circ \delta_{(h,t,1)} \left( (1,1,k) \right) &= &
\delta_{(g,s,1)} \left( \overline{ \lambda_{ \overline{ (h,t,1)}} \left( 
\overline{ (1,1,l) } \right)} \right) \\
&=& \overline{ \lambda_{\overline{(g,s,1)}} \left( \lambda_{ \overline{(h,t,1)}} 
\left( \overline{ (1,1,k) }\right) \right) }\\
&=&\overline{ \lambda_{ \overline{ (h,t,1) \circ (g,s,1) }} \left( \overline{ 
(1,1,l) }\right)}\\
&=& \delta_{(h,t,1) \circ (g,s,1)} (1,1,l).
\end{eqnarray*}

This shows that $\delta: K \longrightarrow \textup{Sym}(R)$, with, for any $h 
\in K$, $\delta(h) = \delta_h$, is a well-defined right action. That $\delta_h$ 
is a bijection follows from the fact that $\delta$ is a homomorphism and that 
$\delta_1 = \textup{id}_R$.

Define for any $(1,1,l) \in R$ and $(g,s,1) \in K$ the map $\sigma_{(1,1,l)}: K 
\longrightarrow K$ with $\sigma_{(1,1,l)}\left( (g,s,1) \right) = \overline{ 
\rho_{\overline{ (1,1,l)}} \left( \overline{ (g,s,1) } \right)}$.
First, we show that $\sigma_{(1,1,1)} = \textup{id}_K$. Let $(g,s,1) \in K$. 
Then,

\begin{align*}
\sigma_{ (1,1,1) } \left( (g,s,1) \right) &= \overline{ \rho_{(1,1,1)} \left( 
\overline{ (g,s,1)} \right)}
=(1,1,1) \circ \left( (g,s,1) (1,1,1) \right)
= (g,s,1).
\end{align*}
Hence, $\delta_{(1,1,1)} = \textup{id}_K$.

Second, we show that $\sigma$ induces an action $\sigma: R \longrightarrow 
\textup{Sym}(K)$. To prove this, let $(1,1,1);(1,1,k) \in R$ and $(g,s,1) \in K$. Then,
\begin{eqnarray*}
\sigma_{(1,1,l)} \circ \sigma_{(1,1,k)} \left( ( g,s,1) \right) &=&
 \sigma_{  (1,1,l)} \left( \overline{ \rho_{ \overline{ (1,1,k)}} \left( \overline{ 
(g,s,1) } \right) } \right)\\
&=& \overline{ \rho_{ \overline{ (1,1,l)}} \circ \rho_{ \overline{ (1,1,k)}} 
\left( \overline{ (g,s,1) } \right)}\\
&=& \overline{ \rho_{ \overline{(1,1,k) \circ (1,1,l)}}\left( 
\overline{(g,s,1)}\right)}\\
&=& \sigma_{(1,1,k) \circ (1,1,l)}\left(g,s,1\right).
\end{eqnarray*}
Thus, $\sigma: R \longrightarrow \textup{Sym}(K)$, with, for any $(1,1,l) \in 
R$, $\sigma \left( (1,1,r) \right) = \sigma_{ (1,1,r)}$, is a well-defined left 
action. Note that $\sigma_{(1,1,r)}$ is a bijection, as $\sigma$ is a 
homomorphism and $\sigma_{(1,1,1)} = \textup{id}_K$.

Let $(1,1,r) \in R$ and $(g,s,1); (h,t,1) \in K$. Then, 
\begin{align*}
&\sigma_{(1,1,r)} \left( ( g,s,1) (h,t,1) \right)= \overline{ 
\rho_{\overline{(1,1,r)}} \left( \overline{ ( gh,s,1) } \right)}\\& =(1,1,r) 
\circ \left( (gh,s,1) \overline{ (1,1,r)} \right) \\ &= (1,1,r) \circ \left( 
(g,s,1) \overline{(1,1,r)} (h,t,1) \overline{(1,1,r)} \right)\\ &= 
\left((1,1,r) \circ \left( (g,s,1) \overline{ (1,1,r) } \right) \right) \left( 
(1,1,r) \circ \left( \overline{ (1,1,r)} (h,t,1)\overline{ 
(1,1,r)}\right)\right).
\end{align*}
Using Proposition \ref{charactanti}, this is equal to

\begin{eqnarray*}
\lefteqn{\left( (1,1,r) \circ \left( (g,s,1) \overline{ (1,1,r) } \right) \right) 
\left( (1,1,r) \circ \left( (h,t,1) \overline{ (1,1,r)} \right)\right)}\\ 
&=&\overline{\rho_{\overline{(1,1,r)}} \left( \overline{ (g,s,1)}\right)} \quad 
\overline{ \rho_{ \overline{ (1,1,r)}} \left( \overline{ (h,t,1)} \right)} 
= \sigma_{(1,1,r)}\left( (g,s,1) \right) \sigma_{ (1,1,r)} \left( (h,t,1) 
\right).
\end{eqnarray*}
Thus, $\sigma_{(1,1,r)}$ is an automorphism of $(K,\cdot)$ for any $(1,1,r) \in 
R$.

Next we  show that both conditions $2$ and $4$ of Definition 
\ref{Matchedproductsemibrace} are satisfied. Let $(1,1,r) \in R$. Then, 

\begin{eqnarray*}
\sigma_{(1,1,r)}\left( (1,1,1) 
\right) &=& \overline{\rho_{ \overline{ (1,1,r) } } \left(\overline{ (1,1,1) } 
\right) } =\overline{  \overline{ \left( (1,1,1) \overline{ (1,1,r)} \right)} 
\circ \overline{ (1,1,r)}}\\ &=& (1,1,r) \circ \overline{ (1,1,r) }= (1,1,1).
\end{eqnarray*}

Let $(g,s,1) \in K$. Then, 
\begin{eqnarray*}
\delta_{(g,s,1)}\left( (1,1,1) 
\right) &=& \overline{ \lambda_{ \overline{ (g,s,1)}} \left( \overline{ (1,1,1)} 
\right)} = \overline{ \overline{ (g,s,1) } \circ \left( (g,s,1) (1,1,1) 
\right) }\\ &=& \overline{ \overline{ (g,s,1)} \circ (g,s,1) }= (1,1,1).
\end{eqnarray*}

Moreover, let $(1,1,r); (1,1,t) \in R$ and $ (g,s,1) \in K$, then we need to 
show that 
\begin{align*}
&\delta_{\sigma_{(1,1,t)} \left( (g,s,1) \right)}\left( (1,1,r) \right) \circ 
\delta_{(g,s,1)} \left( (1,1,t) \right) = \delta_{(g,s,1)}\left( (1,1,r) \circ 
(1,1,t) \right).
\end{align*}
Rewriting the left hand side
\begin{eqnarray*}
\lefteqn{\overline{\left( (1,1,t) \circ \left( (g,s,1) \overline{ (1,1,t)} \right) 
\overline{ (1,1,r)} \right)} \circ (1,1,t) }\\
\lefteqn{
\circ \left( (g,s,1) \overline{ (1,1,t)} \right) \circ \overline{ \left( 
(g,s,1) \overline{ (1,1,t)} \right)} \circ (g,s,1)}\\
&=&\overline{ \left( (1,1,t) \circ \left( (g,s,1) \overline{ (1,1,t)} \right) 
\overline{(1,1,r)} \right)} \circ (1,1,t) \circ (g,s,1)\\
&=&\overline{\left( \overline{ (g,s,1)} \circ \overline{ (1,1,t)} \circ \left( 
(1,1,t) \circ \left( (g,s,1) \overline{ (1,1,t)} \right) \overline{ 
(1,1,r)}\right)\right)}\\
&=&\overline{ \left( \overline{ (g,s,1)} \circ \left( (g,s,1) \overline{ 
(1,1,t)} \left( \overline{ (1,1,t)} \circ \left( \overline{ (1,1,t)} \overline{ 
(1,1,r)}\right)\right)\right)\right)}\\
&=&\overline{ \left( \overline{(g,s,1)} \circ \left( (g,s,1) \left( \overline{ 
(1,1,t)} \circ \overline{ (1,1,r)} \right) \right)\right)}\\
&= &\overline{\lambda_{\overline{(g,s,1)}}\left( \overline{ \left( (1,1,r) \circ 
(1,1,t) \right) } \right)}\\
&=& \delta_{(g,s,1)} \left( (1,1,r) \circ (1,1,t) \right).
\end{eqnarray*}
Analogously, we prove that for any $(g,s,1),(h,t,1) \in K$ and $(1,1,r) \in R$ 
it holds that $$ \sigma_{(1,1,r)}\left( (g,s,1) \right) \circ \left( 
\sigma_{\delta_{(g,s,1)}(1,1,r)}\left( (h,t,1) \right)\right) = 
\sigma_{(1,1,r)} \left( (g,s,1) \circ (h,t,1) \right). $$
Rewriting the left hand side, we get
\begin{eqnarray*}
\lefteqn{(1,1,r) \circ \left( (g,s,1) \overline{ (1,1,r)} \right)\circ \overline{ 
\left( (g,s,1) \overline{ (1,1,r)} \right)} \circ (g,s,1)}\\
\lefteqn{
\circ \left( (h,t,1) 
\left( \overline{ (g,s,1)} \circ \left( (g,s,1) \overline{ (1,1,r)} 
\right)\right)\right)}\\
&=&(1,1,r) \circ (g,s,1) \circ \overline{ (g,s,1) } \circ \left( (g,s,1) \circ 
(h,t,1) \overline{ (1,1,r)} \right)\\
&= &\overline{ \left( \overline{ \left( \left( (g,s,1) \circ (h,t,1) \right) 
\overline{ (1,1,r)} \right)}\right)}\\
&=&\overline{\rho_{\overline{(1,1,r)}}\left( \overline{ (g,s,1) \circ (h,t,1)} 
\right)}\\
&=&\sigma_{(1,1,r)}\left( (g,s,1) \circ (h,t,1) \right).
\end{eqnarray*}
This proves condition $1$ of Definition \ref{Matchedproductsemibrace}.

Finally, let $(1,1,r); (1,1,t) \in R$ and $(g,s,1) \in K$. Then,
\begin{align*}
\overline{ \delta_{(g,s,1)}\left( \overline{(1,1,r) (1,1,t)} \right)} &= 
\lambda_{\overline{(g,s,1)}} \left( (1,1,r) (1,1,t) \right) \\
&= \lambda_{\overline{(g,s,1)}} \left( (1,1,r) \right) 
\lambda_{\overline{(g,s,1)}}\left( (1,1,t) \right)\\
&= \overline{ \delta_{(g,s,1)} \left( \overline{ (1,1,r)} \right) } \quad 
\overline{ 
\delta_{(g,s,1)}\left( \overline{ (1,1,t)} \right)}.
\end{align*}
This shows that condition $5$ from Definition \ref{Matchedproductsemibrace} 
holds. 

We now show that the map $\varphi: K \times R \longrightarrow B$, defined by $$ 
\varphi\left( (g,s,1), (1,1,r) \right) = (g,s,r),$$
is an isomorphism of left semi-braces. First, we check that $\varphi$ is a 
semi-brace homomorphism. 
Let $\left( (g,s,1) , (1,1,r) \right), \left( (h,t,1), 
(1,1,m) \right) \in K \times R$. 
Then,
\begin{eqnarray*}
\lefteqn{\varphi\left( (g,s,1),(1,1,r) \right) \varphi \left( (h,t,1), (1,1,m) 
\right)}\\
&=& (g,s,r) (h,t,m) = (gh,s,m)= \varphi\left( (g,s,1) (h,t,1), (1,1,r) 
(1,1,m)\right)\\&=&\varphi\left( \left( (g,s,1),(1,1,r) \right) \left( (h,t,1), 
(1,1,m) \right)\right).
\end{eqnarray*}

Let $(g,s,1),(h,t,1) \in B1_{\circ}$ and $(1,1,r), (1,1,k) \in R$.
\begin{eqnarray*}
\lefteqn{\varphi\left(\left( (g,s,1), (1,1,r) \right) \circ \left(\left( 
h,t,1\right),\left( 
1,1,k \right)\right) \right)}\\ 
&=& \varphi \left( (g,s,1) \circ 
\sigma_{\overline{\delta_{(g,s,1)}(\overline{(1,1,r)})}}\left( h,t,1\right), 
\left( 1,1,r\right) \circ \left( \overline{ 
\delta_{\overline{\sigma_{\overline{(1,1,r)}}(g,s,1)}}(\left(1,1,k\right))}\right)\right)\\
&=&\varphi\left( (g,s,1) \circ \overline{\left( 
\rho_{\overline{\lambda_{\overline{(g,s,1)}}(1,1,r)}}\left( 
\overline{(h,t,1)}\right)\right)}, (1,1,r) \circ \left( 
\lambda_{\overline{\rho_{(1,1,r)}\left( \overline{(g,s,1)}\right)}}\left( 
1,1,k\right)\right)\right)\\
&=&\varphi\Big( (g,s,1) \circ \overline{(g,s,1)} \circ \left( g,s,r\right) 
\circ \left( (h,t,1) \left(\overline{(g,s,r)} \circ (g,s,1)\right)\right),\\
&&(1,1,r) \circ \overline{(1,1,r)} \circ \left( g,s,r\right) \circ \left(\left( 
\overline{ (g,s,r)} \circ (1,1,r)\right)(1,1,k)\right)\Big)\\
&=&\left( g,s,r\right) \circ \left( (h,t,1) \left(\overline{(g,s,r)} \circ 
(1,1,r) \circ \overline{(1,1,r)} 
\circ \left( g,s,r\right)\right.\right.\\ 
&&\left. \left.\circ \left(\left( \overline{ (g,s,r)} \circ 
(1,1,r)\right)(1,1,k)\right)\right) (g,s,1)\right)\\
&=&(g,s,r) \circ \left( (h,t,1) \left( \overline{(g,s,r)} \circ (g,s,1)\right) 
\left(\overline{(g,s,r)} \circ \left(\left( 
g,s,r\right)(1,1,r)\right)\right)(1,1,r)\right)\\
&=&(g,s,r) \circ \left( (h,t,1) \left( \overline{(g,s,r)} \circ (g,s,1)\right) 
\left(\overline{(g,s,r)} \circ \left( (g,s,r) (1,1,r)\right)\right) (1,1,k) 
\right)\\
&=&(g,s,r) \circ \left( (h,t,1) \left( \overline{(g,s,r)} \circ (g,s,r)\right) 
(1,1,k)\right)\\
&=&(g,s,r) \circ (h,t,k) = \varphi\left(g,s,r\right) \circ  
\varphi\left(h,t,k\right).
\end{eqnarray*}
This shows that $\varphi$ is a left semi-brace homomorphism from $K \bowtie R$ 
into $B$. Clearly, $\varphi$ is bijective. Hence, $B \cong K  \bowtie R$.
\end{proof}



\begin{theorem}
Let $(B,\cdot , \circ)$ be a left semi-brace.   The following conditions are equivalent.
\begin{enumerate}
\item $\rho$ is an anti-homomorphism,
\item $B\cong  ( 1_{\circ} B 1_{\circ} \bowtie E(B1_{\circ}) )) \bowtie E(1_{\circ}B)$ and $E(B)$ is a left  subsemi-brace of $B$.
\end{enumerate}
\end{theorem}

\begin{proof}
Assume $\rho$ is an anti-homomorphism.
Put  $R = E(1_{\circ}B)$, $K = B1_{\circ}$ and $G = 1_{\circ}B1_{\circ}$ a skew left brace. From Theorem~\ref{characttheorem} we know 
	that  $ B \cong (B1_{\circ}) \bowtie R$. Also, by Proposition~\ref{charactanti}, $E(B)$ is a left subsemi-brace.
So in order to show $(2)$, it is sufficient to show that every right 
cancellative completely simple left semibrace $K \cong \mathcal{M}(G, I, 1, \mathcal{I}_{1,I})$ 
can be decomposed as a matched product $1_{\circ}K \bowtie E(K)$. Put $E= E(K)$. 
By Proposition~\ref{charactanti}, this is a left subsemi-brace. 
Every element of $K$ can be 
uniquely written as $sg$, with $s \in E$ and $g \in G$, and thus also $\overline{k}\in E$. Let $g \in G$ and $k 
\in E$.  By Corollary~\ref{GroupIdempotent}, there exists a $t \in E$ such that 
$\overline{k}g = g 
\circ t$. Hence, $$ \rho_g(k) = \overline{ (\overline{k}g)}\circ g = 
\overline{g \circ t } \circ g = \overline{t}. $$ Note that if $g = 1$, then $t= 
\overline{k}$. Hence, $\rho_1 = \textup{id}_E=\rho_{g} \circ \rho_{\overline{g}} =\rho_{\overline{g}} \circ \rho_{g}$. 
Thus, $g \mapsto \rho_g$ defines an  anti-homomorphism $\delta : G 
\longrightarrow \text{Sym} (K)$. 
Put $\delta_g =  \delta (g)$. 
Because of Lemma~\ref{lambdaidempot} and 
Lemma \ref{behaviorlambda} the map $\sigma: E \longrightarrow 
\textup{Aut}(G,\cdot)$, defined by $k \mapsto \lambda_k$ is a homomorphism. We 
will now show that these maps define a matched product of the left semi-braces 
$E$ and $G$. 
By the above, $\sigma$, respectively $\delta$, are  a left, respectively  right,  action on $G$, respectively $E$.
Clearly, for any $g \in G$ and $k \in E$, \begin{align*} \sigma_k(1_{\circ}) &= 
\lambda_k(1_{\circ}) = k \circ (\overline{k}1_{\circ}) = 1_{\circ}, \\ \delta_g(1_{\circ}) &= \rho_g(1_{\circ}) = 
\overline{(\overline{1_{\circ}} g )} \circ g = \overline{g} \circ g = 1_{\circ}. \end{align*}
Let $k \in E$ and $g,h \in G$. Then,
 \begin{align*}
\sigma_k(g) \circ \sigma_{\delta_g(k)}(h) &= \lambda_k(g) \circ 
\lambda_{\rho_{g}(k)}(h) = k \circ ( \overline{k}g) \circ \rho_g(k) \circ ( 
\overline{\rho_g(k)}h) \\ &= k \circ (\overline{k}g) \circ \overline{ 
(\overline{k}g)} \circ g \circ (( \overline{g} \circ (\overline{k}g)) h) \\ &= 
k \circ g \circ ((\overline{g} \circ (\overline{k}g))h) = k \circ ( 
\overline{k}g ( g \circ (\overline{g}h))\\ &= k \circ ( \overline{k}g 
\lambda_g(h))= k \circ ( \overline{k}g \circ(1_{\circ}h)) \\ &= k \circ 
(\overline{k} (g\circ h)) \\ &= \sigma_k(g \circ h)
 \end{align*}
Let $k,t \in E$ and $g \in G$, then 
\begin{align*}
\delta_{\sigma_t(g)}(k) \circ \delta_g(t) &= \rho_{\lambda_{t}(g)}(k) \circ 
\rho_g(t) = \overline{( \overline{k} \lambda_t(g) )} \circ \lambda_t(g) 
\circ \overline{( \overline{t}g)} \circ g\\& = \overline{ (\overline{ k} (t 
\circ \overline{t}g))} \circ t \circ g. \end{align*} 
We need to show that this equals
$ \delta_g(k \circ t) = \overline{ ( \left(\overline{ k\circ 
t}\right) g )} \circ g$.
Hence, we should show that $ \overline{ \left( (\overline{ k \circ t}) g 
\right) } = \overline{ \overline{ k} ( t \circ ( \overline{ t} g ))} \circ t, 
$ or equivalently, 
\begin{eqnarray}
 \overline{t } \circ ( \overline{ k } ( t \circ ( 
\overline{t } g ))) &=& \left(\overline{k \circ t }\right) g. \label{toshow}
\end{eqnarray}
This precisely 
coincides with the demand that $ (\overline{ t } \circ \overline{k}) ( 
\overline{t} \circ ( t \lambda_t(g)) = \left(\overline{k \circ t}\right) g$,
i.e. $$ \left(\overline{k \circ t}\right) \lambda_{\overline{ t}} \lambda_t(g) 
= \left(\overline{k \circ t}\right) g.$$ 
That this equality holds follows from the fact that $ \lambda_{\overline{t}}\lambda_{t} = 
\lambda_{\overline{t} \circ t} = \lambda_{1} = \textup{id}_E$ by
 Lemma~\ref{lambdaidempot}.

Since $E$ is a left zero semi-brace, we obviously get that
$$ \overline{ \delta_g ( \overline{ k t})} = \overline{ \delta_g ( \overline{ 
k})} \overline{ \delta_g ( \overline{ t})}, $$ for all $g \in G$ and $k,t \in 
E$.

Thus, we may construct the matched product $F= G \bowtie E$. It rests to show 
that this is isomorphic to $K$. Clearly, $(K, \cdot) \cong (F,\cdot)$ (under 
the 
natural  map). It rests to show that under the same map $(F,\circ) 
\cong (K, \circ)$. Let $sg, tf \in K$ with $s,t \in E$ and $g,f \in G$. It is 
clear that $ sg = g \circ \overline{ \rho_g( \overline{ s})}$. As $\rho_1 = 
\textup{id}_E$, we get, by replacing $s$ with 
$\overline{\rho_{\overline{g}}(\overline{s})}$, that $$ g \circ s = 
\overline{\rho_{ \overline{ g}}( \overline{s})} g.$$
On the other hand, $sg = s \circ \lambda_{\overline{s}}(g)$, and thus $$ s 
\circ g = s \lambda_s(g).$$
Then, we identify the component of $G$ using the previous remarks,
$$ (sg) \circ (tf) = g \circ \overline{ \rho_g (\overline{s})} \circ f \circ 
\overline{ \rho_f(\overline{t})} = g \circ \lambda_{\overline{ 
\rho_g(\overline{s})}}(f) \circ \underbrace{\dots}_{\in E}.$$ Thus, the group 
component of $(sg)\circ(tf)$ is $g \circ \overline{ \rho_{g} ( \overline{s})}$. 
On the other hand, we identify the component of $E$ using the previous remarks,
$$ (sg) \circ (tf) = s \circ \lambda_{ \overline{s}}(g) \circ t \circ \lambda_{ 
\overline{t}}(f) = s \circ \overline{ \rho_{ \overline{ \lambda_{ 
\overline{s}}(g)}}(\overline{t})} \circ \underbrace{\dots}_{\in G}. $$
Thus the idempotent component of $(sg) \circ (tf)$ is $s \circ \overline{ 
\rho_{ \overline{ \lambda_{ \overline{s}}(g)}}(\overline{t})}$.
So we have shown that $$ (sg) \circ (tf) = \left( s \circ \overline{ \rho_{ 
\lambda_{ \overline{s}}(g)}(\overline{t})} \right) \left( g \circ \lambda_{ 
\overline{ \rho_{g}( \overline{s})}}(f) \right)$$ and 
thus 
    $$ (sg) \circ (tf) = \left( s \circ \overline{ \delta_{\overline{ \sigma_{ \overline{s}}(g)}}(\overline{t})} \right) \left( g \circ \sigma_{ \overline{ \delta_{g}( \overline{s})}}(f) \right). $$
This shows  that the proposed mapping is indeed an isomorphism of left semi-braces. 

Thus, $( K, \cdot, \circ) \cong (G \bowtie E, \cdot, \circ)$. This proves the necessity of (2).

We will now prove the converse. So assume (2) holds.
Because of Proposition \ref{charactanti}, we need to show that the group 
component of $$ (h,k,l) \circ (1,s,t)$$ is also $h$. Clearly, $(h,k,l) \circ 
(1,s,t) = \left( (h,k,l) \circ (1,s,1) \right)\left( 
\lambda_{(h,k,l)}\left((1,1,t)\right)\right).$ As the last term is an 
idempotent by Lemma \ref{lambdaidempot}, it is sufficient to check that 
the group component of $$ (h,k,l) 
\circ (1,s,1)$$ is also $h$.
Since $B$ is a matched product, there exist $h' \in G$ and $k' \in I$ such that 
$$ (h,k,l) = (1,1,l) \circ (h',k',1).$$
We find that $$ (h',k',1) 
\circ (1,s,1) = (h',i,1),$$ for some $i \in I$. Thus, we should show that the 
group components $$ (1,1,l) \circ (h',k',1) \qquad \textup{and} \qquad (1,1,l) 
\circ (h',i,1) $$ coincide. Clearly, we can use the semi-brace property to see 
that $$ (1,1,l) \circ (h',k',1) = \underbrace{ (1,1,l) \circ (1,k',1)}_{\in 
E(B)} 
\lambda_{(1,1,l)}(h,1,1),$$ where $E(B)$ is the set of idempotents of $B$. 
Analogously, $$ (1,1,l) \circ (h',i,1) = \underbrace{(1,1,l) \circ 
(1,i,1)}_{\in E(B)} \lambda_{(1,1,l)}(h,1,1).$$ These elements only differ upto 
an 
idempotent. Hence, it is clear that their group component coincides. 
\end{proof}

\section{Ideals and the Socle of completely simple semi-braces}

 In this section, we introduce the concept of ideals for completely simple left 
 semi-braces $B$. We will assume in this section that the associated $\rho$-map 
 is an anti-homomorphism. At first the definition looks rather technical, 
 although it corresponds with the known definition in case $B$ is a left (skew) 
 brace.
However, we will show that, as desired, this definition corresponds with 
kernels of homomorphisms $f$ of left semi-braces. As for group homomorphisms,
the kernel measures the obstruction to injectivity.

\begin{definition} \label{defideaal}
	Let $B=(\mathcal{M}(G,I,J, \mathcal{I}_{|J|,|I|}),\cdot, \circ)$ be a completely simple left semi-brace 
	such that $\rho$ is an anti-homomorphism. A subsemigroup $I$ of $(B,\cdot)$ 
	is called an ideal if the following conditions are satisfied: 
	\begin{enumerate}
		\item $\left(I \cap G,\cdot\right)$ is a normal subgroup of $(G,\cdot)$,
		\item $\left(I,\circ\right)$ is a normal subgroup of $(B,\circ)$,
		\item $\rho_b(I) \subseteq I$ for all $b \in 1_{\circ}B$,
		\item $\lambda_a(I) \subseteq I$ for all $a \in B1_{\circ}$. 
	\end{enumerate} 
\end{definition}


 This notion corresponds with the notion of ideal defined in 
\cite{catino2017semi} for left cancellative left semi-braces.


\begin{lemma} \label{opsplitsing}
	Let $(B,\cdot,\circ)$ be a completely simple left semi-brace such that 
	$\rho$ is an anti-homomorphism. Let $I$ be an ideal of $B$. If 
	$(g,i,j) \in I$ then  $(g,1,1)$, $(1,i,1)$, $(1,1,j) \in I$.
	\end{lemma}
\begin{proof}
	Let $(g,i,j) \in I$. Clearly, $(1,1,1) \in I$ as $G \cap I$ is a subgroup 
	of $(G,\cdot)$. Hence, because $(I,\cdot)$ is a subsemi-group of 
	$(B,\cdot)$,   $$(g,1,1) = (1,1,1) (g,i,j) (1,1,1) \in I\cap G.$$ This also 
	shows that $(1,i,j) = (g^{-1},1,1) (g,i,j) \in I$. Since $(1,i,1) = 
	(1,i,j) (1,1,1)$ and $(1,1,j) = (1,1,1)(1,i,j)$ we get that 
	$(1,i,1),(1,1,j) \in I$ and thus the result follows.
\end{proof}

Recall \cite{Howie} that an equivalence relation $ \sim $ on 
a semigroup $S$ is called a congruence relation if, for any $s,t,c,d \in S$ with $s \sim t$ and $c \sim d$ then 
 $sc \sim td$. Of course this induces a natural semigroup structure on  $S/\sim$.
Let  $(B,\cdot , \circ )$ be a left semi-brace. An equivalence relation $\sim$ is said to be a congruence on $B$ if $\sim$ is a congruence on
both $(B,\cdot )$ and $(B,\circ )$.

\begin{lemma}\label{rhoidempotisidempot}
	Let $(B,\cdot,\circ)$ be a completely simple left semi-brace such that 
	$\rho$ is an anti-homomorphism. Let $(g,1,1),(1,k,1) \in B$. Then, 
	$$\rho_{(g,1,1)}\left( (1,k,1)\right) \in E(B1_{\circ}).$$
\end{lemma}
\begin{proof}
	As $(E(B1_{\circ}),\circ)$ is a group, it is sufficient to show that $$ 
	\overline{\rho_{(g,1,1)}(\overline{(1,k,1)})} \in E(B1_{\circ}).$$
	Equivalently, we need to show that $$ (g,k,1) = (g,1,1) \circ f,$$ for some 
	$f \in E(B1_{\circ})$.
	First, we show that there exist $(h,1,1) \in B$ and $f \in E(B1_{\circ})$ 
	such that $(g,k,1) = (h,1,1) \circ f$. Denote $\overline{(g,k,1)} = 
	(h,t,1)$. Then,
	\begin{eqnarray*}
	(g,k,1) \circ (1,t,1) &=& \left( (g,k,1) \circ (hh^{-1},t,1) \right)\\  &=& 
	\left( (g,k,1) \circ (h,t,1)\right) \left( 
	\lambda_{(g,k,1)}\left((h^{-1},1,1)\right)\right) \\ &=& 1_{\circ} 
	\lambda_{(g,k,1)}\left((h^{-1},1,1)\right)
	\end{eqnarray*}
	As all terms involved are elements of $B1_{\circ}$, it follows that 
	$(g,k,1) \circ (1,t,1) \in 1_{\circ}B1_{\circ}.$ Hence, there exists a 
	$(h',1,1) \in B$ and $(1,1,f) \in E(B1_{\circ})$ such that $ (g,k,1) = 
	(h',1,1) \circ (1,1,f)$. By Lemma \ref{charactanti}, $h' = g$. This shows 
	the result.
\end{proof}

\begin{prop}
	Let $(B,\cdot, \circ)$ be a completely simple left semi-brace such that 
	$\rho$ is an anti-homomorphism. Let  $I$ be  an ideal of $B$ and let $\sim_I$ be the 
	relation on $B$ given by $$ \textnormal{ for any } x,y \in B: x \sim_I y 
	\Leftrightarrow \overline{y} \circ x \in I.$$
	Then $\sim_I$ is a congruence of the left semi-brace $B$ and hence 
	$B/\sim_{I}$ has a natural   left semi-brace structure. We simply denote 
	this as $B/I$.
\end{prop}
\begin{proof}
Again let $R _{1}= E(1_{\circ}B)$ and 
	$K_{1} = E(B1_{\circ})$; both are left subsemi-braces of $B$ by Theorem~\ref{characttheorem} and Proposition~\ref{charactanti}. 
	Put $F_R= I \cap R_{1}$ and $F_K = I \cap K_{1}$. Clearly, these 
	are normal subgroups of $(R_{1},\circ)$, respectively $(K_{1},\circ)$.
	
	Of course, $\sim_I$ is a congruence for the group $(B,\circ)$. Moreover, $N 
	= 
	I \cap G$ is a normal subgroup of $(G,\circ)$.  It 
	remains to show that $\sim_I$ is a congruence of $(B,\cdot)$. Let $(g,i,j), 
	(h,k,l) \in B$. Because $(h,k,l) =(h,k,j)(1,j,l)$, the semi-brace property yieds
	\begin{equation} 
	\label{firstsplit} \overline{(g,i,j)} \circ (h,k,l) = 
	\left(\overline{(g,i,j)} \circ (h,k,j)\right) \left( \overline{(g,i,j)} 
	\circ (g,i,l)\right).
	\end{equation}
	Since, $\rho_{(1,1,j)}( \overline{(g,i,1)}) = \overline{(g,i,j)} \circ (1,1,j)$ and $\rho_{ (1,1,j)}( \overline{ (h,k,1)}) = \overline{(h,k,j)}\circ (1,1,j)$ 
	we get that 
	\begin{eqnarray*}
	\overline{(g,i,j)} \circ (h,k,j) &=& \rho_{(1,1,j)}( \overline{(g,i,1)}) \circ \overline{(1,1,j)} \circ (h,k,j) \\ 
	&= &\rho_{(1,1,j)}(\overline{(g,i,1)}) \circ \overline{ \overline{ (h,k,j)} \circ (1,1,j) } \\ 
	&= &\rho_{ (1,1,j)}(\overline{(g,i,1)}) \circ \overline{ \rho_{ (1,1,j)}( \overline{ (h,k,1)})}. 
	\end{eqnarray*}
	Because $B1_{\circ}$ is a 
	subsemi-brace and because of  Lemma \ref{rhomap1},
	\begin{eqnarray}
	\overline{(g,i,j)} \circ (h,k,j) &\in& B1_{\circ}. \label{containedin1}
	\end{eqnarray}
	Furthermore, by Lemma \ref{lambdaidempot}, 
	\begin{eqnarray}
	 \overline{ (g,i,j)} \circ 
	(g,i,l) = \lambda_{\overline{ (g,i,j)}}(1,1,l) &\in & R_{1}. \label{containedin2}
	\end{eqnarray}
	
	We next prove two equivalences (\ref{step1}) and (\ref{step2}).
	First we show, for any $(g,i,j)$ and  $(h,k,l) \in B$,
	\begin{eqnarray} 
	\label{step1} \hspace{0,5cm}(g,i,j) \sim_I (h,k,l) &\textnormal{ if and only if }& (g,i,1) 
	\sim_I (h,k,1) \textnormal{ and } (1,1,j) \sim_I (1,1,l).
	\end{eqnarray} 
	
	Indeed, $(h,k,l) \sim_I (g,i,j)$ means that
	$\overline{ (g,i,j)} \circ (h,k,l) \in I$. Hence, by (\ref{containedin1}) and (\ref{containedin2}), 
	$$ 
	\overline{(g,i,j)} \circ (h,k,j) \in I \cap B1_{\circ} \textnormal{ and } 
	\overline{(g,i,j)} \circ (g,i,l) \in F_R.
	$$
	Because $\lambda$ is a homomorphism we have that 
\begin{eqnarray}
	\overline{(g,i,j)} \circ (g,i,l) &=& \lambda_{ \overline{ (g,i,j)}}(1,1,l) = 
	\lambda_{ \rho_{ (1,1,j)} (\overline{ (g,i,1)} \circ \overline{ 
	(1,1,j)})}(1,1,l) \nonumber \\ 
	&=& \lambda_{ \rho_{ (1,1,j)}( \overline{ (g,i,1)})} 
	\left( \lambda_{ \overline{ (1,1,j)}} ( 1,1,l) \right)  \nonumber \\ 
	&= &
	\lambda_{\rho_{ (1,1,j)} ( \overline{ (g,i,1)} )}\left( \overline{(1,1,j)} 
	\circ (1,1,l) \right). \label{lambdain}
	\end{eqnarray}
	Because of Lemma~\ref{rhomap1}, $\rho_{ (1,1,j)} ( \overline{ (g,i,1)})\in 
	B1_{\circ}$. Hence, 
	by Definition \ref{defideaal}, 
	$$  \overline{ (1,1,j)} 
	\circ (1,1,l) = \lambda_{ \overline{ \rho_{ (1,1,j)} (\overline{ (g,i,1)}) 
	}} \left(  \overline{( 
	g,i,j)} \circ (g,i,l) \right) \in I, 
	$$ 
	i.e. $(1,1,l) \sim_I (1,1,j)$. 
	
	Now
	\begin{eqnarray*}
	\overline{(g,i,j)} \circ (h,k,j) 
	&=& \left( \overline{(g,i,j)} \circ (h,k,1)\right) \left( \overline{(g,i,j)} 
	\circ (g,i,j)\right)\\
	&= &\left( \overline{(g,i,j)} \circ (h,k,1)\right)1_{\circ}\\
	&=&\overline{ \left( \overline{ (h,k,1)} \circ (g,i,j)\right)}1_{\circ}
	\end{eqnarray*}
         So, by   the semi-brace property, 
	\begin{eqnarray*}
	\lefteqn{\overline{(g,i,j)} \circ (h,k,j)}\\
	&=&\overline{ \left( \overline{ (h,k,1)} \circ (g,i,1)\right) \left( 
	\overline{(h,k,1)} \circ (h,k,j)\right)}1_{\circ}\\
	&=&  \left( \overline{ \left( \overline{ (h,k,1)} \circ (g,i,1)\right) \left( 
	\overline{(h,k,1)} \circ (h,k,j)\right)} \right.\\
	&& \left.  \circ \overline{(h,k,1)} \circ 
	(h,k,j) \circ \overline{ (h,k,j)} \circ (h,k,1)\right) 1_{\circ}\\
%
	&=&\left( \rho_{\overline{(h,k,1)} \circ (h,k,j)} \left( \overline{ (g,i,1)} 
	\circ (h,k,1)\right) \circ \overline{ (h,k,j)} \circ (h,k,1)\right)1_{\circ}\\
	&=&\left( \rho_{\overline{(h,k,1)} \circ (h,k,j)} \left( \overline{ (g,i,1)} 
	\circ (h,k,1) \right) \circ \overline{ \overline{ (h,k,1)} \circ (h,k,j)} 
	\right) 1_{\circ}\\
	&=& \left( \rho_{\overline{(h,k,1)} \circ (h,k,j)} \left( \overline{ 
	(g,i,1)} \circ (h,k,1) \right) \circ \left( \overline{ 
	\lambda_{\overline{(h,k,1)}} (1,1,j)}\right)\right)1_{\circ}
	\end{eqnarray*}
	Put $x=\rho_{\overline{(h,k,1)} \circ (h,k,j)} \left( \overline{ 
	(g,i,1)} \circ (h,k,1) \right)$ and $y = \overline{ 
	\lambda_{\overline{(h,k,1)}} (1,1,j)}$. By Lemma~\ref{rhomap1}, $y \in E(1B)$ and $x\in B1_{\circ}$.
	Hence,   the previous equality becomes,
        \begin{eqnarray*}
	\overline{(g,i,j)} \circ (h,k,j) &=& \left(x \circ (1_{\circ}y)\right) 1_{\circ}
	= x \lambda_x(y) 1_{\circ}
	= x 1_{\circ} \\
	  &=&x =
	\rho_{\overline{(h,k,1)} \circ (h,k,j)} \left( \overline{ (g,i,1)} \circ 
	(h,k,1) \right)
	\end{eqnarray*}
	 We thus  get 
	that 
	\begin{eqnarray} \overline{(g,i,1)} \circ (h,k,1)& =&\rho_1 ( 
	\overline{(g,i,1)} \circ (h,k,1)) \nonumber \\ 
	&= &\rho_{\overline{\overline{ (h,k,1)}  \circ (h,k,j)}} \rho_{\overline{ (h,k,1)} \circ (h,k,j)} (\overline{ 
	(g,i,1) } \circ (h,k,1)) \nonumber \\ 
	&=& \rho_{\overline{(h,k,j)}\circ (h,k,1)}\left( 
	\overline{ (g,i,j)} \circ (h,k,j)\right) . \label{rhoequality}
	\end{eqnarray}
	Because  $\overline{(g,i,j)} \circ (h,k,j) \in I$, and since $\overline{(h,k,j)}\circ (h,k,1) \in 1_{\circ}B$ by (\ref{containedin2}),
         we thus obtain that  $(h,k,1) \sim_{I} (g,i,1)$. So, we have proved the necessity of (\ref{step1}).
	
	That these conditions are sufficient also easily follows from some of the 
	above obtained equalities.
	Indeed, because of (\ref{rhoequality}), if $\overline{(g,i,1)} \circ 	
	(h,k,1) \in I$ then  $\overline{ 
	(g,i,j)} \circ (h,k,j) \in I$.
	Because of (\ref{lambdain}), if $\overline{(1,1,j)} \circ (1,1,l) \in I$ 
	then $\overline{(g,i,j)} \circ (g,i,1) \in I$.
	Hence, because of (\ref{firstsplit}), the conditions $\overline{(g,i,1)} 
	\circ 	(h,k,1) \in I$ and 
	$\overline{(1,1,j)} \circ (1,1,l) \in I$ imply that $\overline{(g,i,j)} \circ (h,k,l)$. This finishes the proof of the sufficiency.
	
	
	We now prove the second promised equivalence:
	 \begin{eqnarray} 
	\label{step2} (g,i,1) \sim_I (h,k,1) \textnormal{ if and only if } (g,1,1) 
	\sim_I 
	(h,1,1) \textnormal{ and } (1,i,1) \sim_I (1,k,1). 
	\end{eqnarray}
	Indeed, 
	suppose $(g,i,1) \sim_I (h,k,1)$, that is  $\overline{(g,i,1)}\circ (h,k,1) 
	\in 
	I$. As $(I,\circ )$ is a normal subgroup of $(B,\circ)$, it follows that 
	$$
	\lambda_{\overline{(1,i,1)}}(g,1,1) \circ \overline{(g,i,1)}\circ (h,k,1) 
	\circ \overline{\lambda_{\overline{(1,i,1)}}(g,1,1)} \in I. 
	$$ 
	Put $x=\lambda_{\overline{(1,k,1)}}(h,1,1) \circ  \overline{\lambda_{\overline{(1,i,1)}}(g,1,1)}$.
	Because of Lemma~\ref{rhomap1}, $x\in 1_{\circ}B$.
	It follows that
	{\small \begin{eqnarray*} 
	\lefteqn{
	\lambda_{\overline{(1,i,1)}}(g,1,1) \circ 
	\overline{(g,i,1)}\circ (h,k,1) \circ 
	\overline{\lambda_{\overline{(1,i,1)}}(g,1,1)} }\\
	&=&\lambda_{\overline{(1,i,1)}}(g,1,1) \circ \overline{\overline{(1,i,1)} 
	\circ (g,i,1)} \circ \overline{(1,i,1)} \\
	&&\circ (1,k,1) \circ 
	\overline{(1,k,1)} \circ (h,k,1) \circ \overline{ 
	\lambda_{\overline{(1,i,1)}}(g,1,1)}\\
	&=& \lambda_{\overline{(1,i,1)}}(g,1,1) \circ 
	\overline{\lambda_{\overline{(1,i,1)}}(g,1,1)} \circ 
	\overline{(1,i,1)} \circ (1,k,1) \circ 
	\lambda_{\overline{(1,k,1)}}(h,1,1) \circ 
	\overline{\lambda_{\overline{(1,i,1)}}(g,1,1)}\\
	&=& \overline{(1,i,1) } \circ (1,k,1) \circ (1_{\circ}x) \\
	&=& \left(\overline{(1,i,1) } \circ (1,k,1) \right)
	\lambda_{\overline{(1,i,1) } \circ (1,k,1)}(x), 
	\end{eqnarray*}
	}
	Now, $\overline{(1,i,1) } \circ (1,k,1) \in K_1$ and $\lambda_{\overline{(1,i,1) } \circ (1,k,1)}(x)$. Hence,
	it follows from Lemma~\ref{opsplitsing}  that
	$\overline{(1,i,1) } \circ (1,k,1)\in I$, i.e.  $(1,i,1) \sim_I 	(1,k,1)$. 
	
	Moreover,  since $(I,\circ)$ is a normal subgroup of 
	$(B,\circ)$ , 
   \begin{eqnarray*}
	\lefteqn{\overline{ \rho_{(g,1,1)}\left(\overline{(1,i,1)}\right)} \circ 
	\overline{(g,i,1)} \circ (h,k,1) 
	\circ \rho_{(g,1,1)}\left( \overline{(1,i,1)}\right)} 	 \\
	&=&\overline{\rho_{(g,1,1)} \left( \overline{(1,i,1)}\right)} \circ 
	\overline{(g,i,1)} \circ (g,1,1) \circ \overline{(g,1,1)}\\
	&& \circ (h,1,1) 
	\circ \overline{(h,1,1)} \circ (h,k,1) \circ \rho_{(g,1,1)}\left( 
	\overline{(1,i,1)}\right)\\
	&=&  \overline{ \rho_{(g,1,1)}\left(\overline{(1,i,1)}\right)} \circ 
	\rho_{(g,1,1)}\left(\overline{(1,i,1)}\right) \circ  
	\overline{ (g,1,1)}\\
	&& \circ (h,1,1) \circ \overline{ \rho_{ (h,1,1)} (1,k,1)} 
	\circ \rho_{(g,1,1)}\left(\overline{(1,i,1)}\right) \\
	&=& \overline{ (g,1,1)} \circ (h,1,1) \circ 
	\overline{ \rho_{ (h,1,1)} (1,k,1)} \circ 
	\rho_{(g,1,1)}(1,i,1) \in I.
	\end{eqnarray*}
	By Proposition \ref{charactanti} and Lemma \ref{rhoidempotisidempot},
	$ \overline{ (g,1,1)} \circ (h,1,1) \in G$ and $\overline{ \rho_{ (h,1,1)} (1,k,1)} \circ 
	\rho_{(g,1,1)}(1,i,1)\in E(B1_{\circ})$.
	  By Corollary \ref{GroupIdempotent}, multiplying (in $(B,\circ)$) on the 
	  right by an element of $E(B1_{\circ}$ does not change the group component 
	  of an element of $B1_{\circ}$.	
           Hence, by 
	Lemma \ref{opsplitsing}, it follows that $(g,1,1) \sim_I (h,1,1)$.
	This proves the necessity of the conditions in (\ref{step2}).

	We now  show that they these condition are sufficient. So, suppose  $(1,i,1) \sim_I (1,k,1)$ and 
	$(g,1,1) 
	\sim_I (h,1,1)$. Then, as $I$ is a normal subgroup of $(B,\circ)$, 
{\small 	\begin{align*}
	&\overline{(g,i,1)} \circ (h,k,1) \\ 
	=& \overline{\lambda_{\overline{(1,i,1)}}(g,1,1)} \circ \overline{(1,i,1)} 
	\circ (1,k,1) \circ \lambda_{ \overline{ (1,k,1)}}(h,1,1) \\
	=& \underbrace{\overline{\lambda_{\overline{(1,i,1)}}(g,1,1)} \circ 
	\overline{(1,i,1)} \circ (1,k,1) \circ 
	\lambda_{\overline{(1,i,1)}}(g,1,1)}_{\in I} \circ 
	\underbrace{\overline{\lambda_{\overline{(1,i,1)}}(g,1,1)} \circ \lambda_{ 
	\overline{ (1,k,1)}}(h,1,1)}_{=x \in 1_{\circ}B1_{\circ}} \\ 
	=& \underbrace{\overline{\lambda_{\overline{(1,i,1)}}(g,1,1)} \circ 
	\overline{(1,i,1)} \circ (1,k,1) \circ 
	\lambda_{\overline{(1,i,1)}}(g,1,1)}_{=t\in I \cap B1_{\circ}} 
	\lambda_{t}(x),
	\end{align*}
	}
	where the last equality holds by $(\ref{lambdaproduct})$ for some $r \in 
	B$. Both $t$ and $\overline{(g,i,1)} \circ (h,k,1)$ are elements in 
	$B1_{\circ}$. Hence, there exist $\alpha,\beta \in E(B1_{\circ})$ and $ a,b 
	\in 1_{\circ} B 1_{\circ}$ such that $ t = \alpha a$ and 
	$\overline{(g,i,1)} \circ (h,k,1)=\beta b$.
	Since $\lambda_t(x) \in 1_{\circ}B1_{\circ}$, it is clear from the previous 
	calculation that $\alpha = \beta$. Hence, it follows by 
		Lemma 
		\ref{opsplitsing} that $\alpha$ is 
	contained in $I$. On the other hand, again because $I$ is a normal subgroup 
	of $(B,\circ)$, \begin{align*}
	\overline{(g,i,1)} \circ (h,k,1) 
	= \rho_{(g,1,1)}\left(\overline{(1,i,1)}\right) \circ \overline{(g,1,1)} 
	\circ (h,1,1) \circ 
	\overline{\rho_{(h,1,1)}\left(\overline{(1,k,1)}\right)},
	\end{align*}
	 we get, by Lemma \ref{rhoidempotisidempot}, that 
	{\small \begin{align*}
	 \underbrace{\rho_{(g,1,1)}\left( \overline{(1,i,1)}\right) \circ 
	 \overline{(g,1,1)} \circ (h,1,1) 
	\circ \overline{ \rho_{(g,1,1)}\left(\overline{(1,i,1)}\right)}}_{ \in I 
	\cap B1_{\circ}} \circ 
	\underbrace{\rho_{(g,1,1)}(1,i,1) \circ \overline{\rho_{(h,1,1)}(1,k,1)}}_{ 
	\in E(B1_{\circ}}.
	\end{align*}
	}
	By Corollary \ref{GroupIdempotent}, the group component of an element 
	remains 
	fixed under right multiplication with an idempotent of $B1_{\circ}$. Thus, 
	by Lemma 
	\ref{opsplitsing}, the group 
	component of $\overline{(g,i,1)} \circ (h,k,1)$ is contained in $I$. Since 
	$(I,\cdot)$ is a subsemi-group of $(B,\cdot)$ we get that 
	$\overline{(g,i,1)} 
	\circ (h,k,1) \in I$. Hence, we have proven the equivalence (\ref{step2}).
	
	We are now in a position to finish the proof by prove that  $\sim_I$ is a congruence on 	$(B,\cdot)$. 
	Let $(g,i,j) \sim_I (f,l,m)$ and $(h,t,s) \sim_I (d,n,p)$. Then, by  the equivalences (\ref{step1}) and (\ref{step2}), 
	 \begin{eqnarray}
	(1,1,s) \sim_I (1,1,p), && 
	(g,1,1) \sim_I (f,1,1) \label{eqskew1}\\
	(h,1,1) \sim_I (d,1,1),&& 
	(1,i,1) \sim_I (1,l,1). \label{eqskew2}
	\end{eqnarray}
	We show that $I \cap 1_{\circ} B 1_{\circ}$ is an ideal of the skew left 
	brace $1_{\circ}B1_{\circ}$ (see Proposition \ref{rowcol}). As $(I,\circ)$ 
	is a normal subgroup of $(B,\circ)$, it follows that $(I \cap 
	1_{\circ}B1_{\circ}, \circ)$ is a normal subgroup of 
	$(1_{\circ}B1_{\circ},\circ)$. Furthermore, by definition of ideals of left 
	semi-braces, $(I \cap 1_{\circ}B1_{\circ},\cdot)$ is a normal subgroup of 
	$(1_{\circ}B1_{\circ},\cdot)$. Lastly, it is clear, by definition of an 
	ideal of a left semi-brace, that $\lambda_a(I\cap 1_{\circ}B1_{\circ}) \subseteq I\cap 1_{\circ}B1_{\circ}$ for some $ a 
	\in 1_{\circ}B1_{\circ}$. 
	Thus, $I\cap 1_{\circ}B1_{\circ}$ is an ideal of skew left braces.
	It follows from the (\ref{eqskew1}) and (\ref{eqskew2})  that $$ (gh,1,1) 
	\sim_I 
	(fd,1,1).$$ 
	So, by the equivalences (\ref{step1}) and (\ref{step2}) and by (\ref{eqskew1}) and (\ref{eqskew2}),
	 $$ (g,i,j)(h,t,s)=(gh,j,s) \sim_I (fd,l,p)=(f,l,m)(d,n,p).$$
	Thus, $\sim_I$ is a congruence on $(B,\cdot)$.
\end{proof}

\begin{definition}
	Let $(B,\cdot,\circ)$ and $ (B_1,\cdot,\circ)$ be left semi-braces.
	The kernel of a  a left semi-brace homomorphism $f:B \rightarrow B_1$ is 
	defined as $$ \ker f = \left\lbrace a \in B \mid f(a) = 1_{\circ} \right\rbrace.$$
\end{definition}

\begin{corollary}
	Let $(B,\cdot,\circ)$ be a completely simple left semi-brace such that 
	the associated $\rho$-map is an anti-homomorphism. 
	The ideals of $B$ are precisely the kernels of left semi-brace homomorphisms with domain $B$.
\end{corollary}
\begin{proof}Put $G = 1_{\circ}B1_{\circ}$. First we show that a kernel of a 
left semi-brace homomorphism
	$f:B \longrightarrow B_1$ is an ideal. 
	With well-known elementary arguments one shows that  $\ker f$ is a normal 
	subgroup of $(B_1,\circ)$
	and that $(\ker f,\cdot )$ is a subsemigroup of $(B,\cdot )$ such  that  $(\ker f \cap G, \cdot)$ 
	is a normal subgroup of $(G,\cdot )$.
	Let $b \in 1_{\circ}B$, then 
	\begin{align*}
	f(\rho_b(x)) &= f(\overline{ (\overline{x}b)} \circ b)
	=\overline{ \overline{f(x)} f(b)} \circ f(b) 
	=\overline{1_{\circ}f(b)} \circ f(b) 
	= \overline{f(b)} \circ f(b) = 1_{\circ}.
	\end{align*}
	Hence $\rho_{b}(\ker f) \subseteq \ker f$.
	Analogously, one can shows that
	$f(\lambda_a(x)) = 1_{\circ}$, for $a \in B 1_{\circ}$.
	Thus, $\lambda_{a}(\ker f) \subseteq \ker f$. So,  $\ker f$ is an ideal of $B$.
	
	On the other hand, let $I$ be an ideal of $B$. Define $f: B 
	\longrightarrow B/I$. Then $f: B \rightarrow B/I$ is a left  semi-brace 
	morphism with 
	$I = \ker f$.
\end{proof}
\begin{definition}
	Let $(B,\cdot,\circ)$ be a completely simple left semi-brace.
	 The socle of $B$ is 
	 $$ \textup{Soc}(B)= \left\lbrace x \in B \mid \lambda_x = 
	\lambda_{1_{\circ}} \textnormal{ and } \rho_x = \rho_{1_{\circ}} \right\rbrace.$$
\end{definition}

\begin{prop}\label{charactersocl}
	Let $(B,\cdot,\circ)$ be a completely simple left semi-brace such that the 
	associated $\rho$-map is an anti-homomorphism. Then, the following 
	propositions are 
	equivalent:
	\begin{enumerate}
		\item $x \in \textup{Soc}(B)$,
		\item for any $b \in B$, $\overline{x} \circ (1_{\circ}b) = 
		\overline{x}b$ and $x \circ (b1_{\circ}) = bx$.
	\end{enumerate}
	Moreover, $\textup{Soc}(B) \subseteq \textup{Soc}(1_{\circ}B1_{\circ})=\{ x\in  1_{\circ}B1_{\circ} \mid
	 x\circ y =xy =yx \text{ for all } y\in 1_{\circ}B1_{\circ}\}$.
\end{prop}
\begin{proof}
	First, suppose $x \in \textup{Soc}(B)$. Thus, for any $b\in B$, 
	$x\circ (\overline{x}b) =\lambda_x(b) = \lambda_{1_{\circ}}(b) = 1_{\circ}b$,
	and thus
	$\overline{x} b = \overline{x} \circ (1_{\circ}b)$.
	Also,
	$\overline{bx}\circ x = \rho_x(\overline{b}) = \rho_{1_{\circ}}(\overline{b}) = 
	\overline{b1_{\circ}}$,
	and thus
	$\overline{bx} = \overline{b1_{\circ}} \circ \overline{x}$. 
	So
	$bx = x \circ (b1_{\circ})$.
	Thus (1) implies (2). Reversing the 
	reasoning, the converse is easily verified.
	
	Let $x \in \textup{Soc}(B)$. Then, 
	$ 1_{\circ} = \lambda_{1_{\circ}}(1_{\circ}) = \lambda_x(1_{\circ}) = x 
	\circ (\overline{x}1_{\circ})$. Thus, $\overline{x} = \overline{x}1_{\circ}$. This shows 
	that $\overline{x} $ and thus also $x\in B1_{\circ}$,
	as $B1_{\circ}$ is a left subsemi-brace by Theorem~\ref{characttheorem}. 
	Furthermore, $ 1_{\circ} = \rho_{1_{\circ}}(1_{\circ}) = \rho_x(1_{\circ}) 
	= \overline{1_{\circ}x} 
	\circ x$. Thus, $\overline{x} = \overline{1_{\circ}x}$. As $1_{\circ}B$ is 
	a subsemi-brace by Corollary \ref{rowcol}, we thus get that, $x \in 
	1_{\circ}B$. It 
	follows that $x \in 1_{\circ}B1_{\circ}$ and thus $x \in 
	\textup{Soc}(1_{\circ}B1_{\circ})$. 
\end{proof}

\begin{prop}
	Let $(B,\cdot, \circ)$ be a completely simple left semi-brace such that the 
	associated $\rho$-map is an anti-homomorphism. Then, $\textup{Soc}(B)$ is 
	an ideal of $B$.
\end{prop}
\begin{proof}
	Let $x,y \in \textup{Soc}(B)$ and let $z\in B$.
	Because $\lambda$ is a homomorphism by Lemma~\ref{lambdaidempot}
	 $$
	\lambda_{\overline{z}\circ x\circ y \circ z} 
	= \lambda_{\overline{z}} \circ \lambda_{x} \circ \lambda_{y}
	 \circ \lambda_{z} =\lambda_{\overline{z}} \circ \lambda_{1_{\circ}} \circ 
	\lambda_{1_{\circ}}\circ \lambda_{z} = \lambda_{\overline{z}}\circ \lambda_{1_{\circ}} \circ  \lambda_{z}
	= \lambda_{1_{\circ}}.
	$$
	  Similarly, since by assumption, $\rho$ is an anti-homomorphism, 
	$\rho_{\overline{z} \circ x \circ y \circ z} = \rho_{y} \circ \rho_x = \rho_{1_{\circ}} \circ \rho_{1_{\circ}} = \rho_{1_{\circ}}
	$.
	Moreover,
         $ \rho_{1_{\circ}} = \rho_{x \circ \overline{x}} = \rho_{\overline{x}} 
	\circ 
	\rho_x = \rho_{\overline{x}} \circ \rho_{1_{\circ}} = 
	\rho_{\overline{x}}$
	and, similarly, $\lambda_{\overline{x}} =\lambda_{1_{\circ}}$.
	It follows that 
	$\textup{Soc}(B)$ is a normal subgroup of $(B,\circ)$.
	
	Let $x,y \in \textup{Soc}(B)$ and $b \in B$. By Proposition \ref{charactersocl}, $\textup{Soc}(B) 
	\subseteq \textup{Soc}(1_{\circ}B1_{\circ})$ and thus 
	$
	xyb = x \circ (1_{\circ}yb) = x \circ (yb) = x \circ y \circ (1_{\circ}b) = (xy) \circ (1_{\circ}b)
	$
	and
        $ 
	bxy = byx = x \circ (by1_{\circ}) = x \circ (by) = x \circ y \circ (b1_{\circ}) = (xy) 
	\circ (b1_{\circ}).
	$
	Thus, again by Proposition \ref{charactersocl}, $xy\in \textup{Soc}(B)$ and hence
	$\textup{Soc}(B)$  is a subsemi-group of $(B,\cdot )$.
	
	Let $x \in \textup{Soc}(B)$.  By Proposition 
	\ref{charactersocl}, and as $I$ is a normal subgroup of $(B,\circ)$, 
	if $b \in 1_{\circ}B$ then
	$
	\rho_b(x) = \overline{(\overline{x}b)} \circ b = \overline{ \overline{x} 
	\circ (1_{\circ}b)} \circ b = \overline{b} \circ x \circ b \in I
	$;
	and if 
        $b \in B1_{\circ}$ then
	$ \lambda_b(x) = b \circ (\overline{b} x) = b \circ x \circ 
	(\overline{b}1_{\circ}) = b \circ x \circ \overline{b} \in I$.
	It follows that $\textup{Soc}(B)$ is an ideal of $B$.
\end{proof}

 Proposition~\ref{charactersocl} shows that for a skew left brace $(B,\cdot , 
\circ )$ (Definition \ref{skewbrace}) corresponds
with the definition given by Guarnieri and Vendramin in 
\cite{guarnieri2017skew}:
$$ \textup{Soc}(B) = \{ x\in  B \mid
x\circ y =xy =yx \text{ for all } y\in B\}.$$

	
	\section{Solutions to the Yang-Baxter equation and their structure monoids}
	
	 In this section we first show a completely simple left semi-brace $B$, 
	 with 
	associated $\rho$-map an involution, gives rise to  a set-theoretic solution 
	of 
	the Yang-Baxter equation. We next introduce the structure monoid associated 
	to
	this solution, denoted $M(B)=M(r_B)$. In case $B$ is a left or right zero 
	left 
	semi-brace we prove that $M(B)$ has a group of fractions.
	
	The proof of the following theorem is in essence the proof of Theorem 9 
	given 
	by Catino, Colazzo and Stefanelli in \cite{catino2017semi}.
	For completness' sake and to show where we use that $\rho$ is an 
	anti-homomorphism we include a proof.
	Without specific reference we will often make use of the following fact, 
	for $a$ and $b$ in a left semi-brace $(B,\cdot , \circ )$:
	$$a \circ b = a \circ \left( \overline{a} b \right) \circ \overline{\left( 
	\overline{a} b \right)} \circ b =
	\lambda_a(b) \circ \rho_{b}(a).$$
	
	\begin{theorem}
		Let $(B,\cdot,\circ)$ be a left semi-brace such that $\rho: (B,\circ) 
		\longrightarrow \textup{Map}(B,B)$ is an anti-homomorphism. Then, the 
		mapping $r_{B}: B \times B \longrightarrow B \times B$ given by $r(x,y) 
		= 
		(\lambda_{x}(y), \rho_{y}(x))$ is a set-theoretic solution of the 
		Yang-Baxter equation, called the solution associated to $B$. 
		Furthermore, 
		if $(B,\cdot,\circ)$ is a left semi-brace such that 
		$(1_{\circ}B1_{\circ},\cdot,\circ)$ is a left brace, then $r_B^3=r_B$. 
		Moreover, if $(B,\cdot,\circ)$ is a left semi-brace such that 
		$1_{\circ}B1_{\circ} = \left\lbrace 1_{\circ} 
		\right\rbrace$, then $r_B^2 = r_B$.  
	\end{theorem}
	\begin{proof}
	It is easily verified that $(B,r_B)$ is a solution of the Yang-Baxter 
	equation if and only if, for all $x,y,z \in B$,
		\begin{eqnarray*}  \lefteqn{\left( \lambda_x \lambda_y(z), 
		\lambda_{\rho_{\lambda_y(z)}(x)}(\rho_z(y)), 
		\rho_{\rho_{z}(y)}(\rho_{\lambda_{y}(z)}(x)) \right)}\\
		& =& 
		\left(\lambda_{\lambda_x(y)}( \lambda_{\rho_y(x)}(z)),
		\rho_{\lambda_{\rho_y(x)}(z)}(\lambda_{x}(y)), \rho_z \rho_y(x) 
		\right). \\
		\end{eqnarray*}
		Denote the first triple by $(s_1,s_2,s_3)$ and the second by 
		$(t_1,t_2,t_3)$. 
		Then,
		$$ 
		t_1 \circ t_2 \circ t_3 = \lambda_x(y) \circ \lambda_{\rho_y(x)}(z) 
		\circ \rho_z \rho_y(x) = 
		\lambda_x(y) \circ \rho_y(x) \circ z = x \circ y \circ z.
		$$
		and
		$$ 
		s_1 \circ s_2 \circ s_3 = \lambda_x \lambda_y(x) \circ 
		\rho_{\lambda_{y}(z)}(x) \circ \rho_z(y) = 
		x \circ \lambda_y(z) \circ \rho_z(y) = x \circ y \circ z. 
		$$
		Thus, $ t_1 \circ t_2 \circ t_3 = s_1 \circ s_2 \circ s_3$.
		We will now show that $t_1 = s_1$ and $t_3 = s_3$. As $\lambda: (B, 
		\circ) 
		\longrightarrow \textup{End}(B,\cdot)$ is a homomorphism by 
		Lemma~\ref{lambdaidempot}, 
		it follows that 
		$$ t_1 = \lambda_{\lambda_x(y) \circ \rho_{y}(x)}(z) = \lambda_{x \circ 
			y}(z) = \lambda_x \lambda_{y}(z) = s_1.$$
		Furthermore, as, by assumption, $\rho$ is an anti-homomorphism, 
		$$ 
		s_3 = \rho_{\lambda_y(z)\circ \rho_z(y)}(x) = \rho_{y \circ z}(x) = 
		\rho_z \rho_y (x) = t_3.
		$$
		Because $(B,\circ)$ is a group, it thus also  follows that $s_2 = t_2$.
		
		Let $(g,i,j), (h,k,l) \in B$. Clearly, 
		 $$r_B((g,i,j),(h,k,l)) = \left( 
		\lambda_{(g,i,j)}(h,k,l),\rho_{(h,k,l)}(g,i,j)\right) .$$ Write 
		$r_B^2((g,i,j),(h,k,l)) = (u,v)$. Then 
		\begin{eqnarray*}
			u&=&\lambda_{\lambda_{(g,i,j)}(h,k,l)}\left( 
			\rho_{(h,k,l)}(g,i,j)\right)\\ 
			&=& 
			\lambda_{(g,i,j)}(h,k,l) \circ \left( \overline{ 
			\lambda_{(g,i,j)}(h,k,l)} 
			\rho_{(h,k,l)}(g,i,j)\right) \\
			&=& (g,i,j) \circ \left(\overline{(g,i,j)}(h,k,l)\right)\\
			&& \circ 
			\left(\left( 
			\overline{\left( \overline{(g,i,j)}(h,k,l)\right)} \circ 
			\overline{(g,i,j)}\right) \left( \overline{ 
			\overline{(g,i,j)}(h,k,l)} 
			\circ (h,k,l) \right) \right).
		\end{eqnarray*}
		Put $t =\overline{ \overline{(g,i,j)}(h,k,l)}$ and $(g_2,f,q) = 
		\overline{(g,i,j)}$. 
		Hence,
		\begin{eqnarray*}
			u&=&(g,i,j) \circ \overline{t} \circ \left(\left( t \circ 
			\overline{(g,i,j)}\right) \left( t \circ (h,k,l)\right)\right)\\
			&=& (g,i,j) \circ \overline{t} \circ \left(\left( t \circ 
			\overline{(g,i,j)}\right) \left(t \circ (1,k,1)\right) 
			\lambda_t(h,1,l)\right)\\
			&=& (g,i,j) \circ \overline{t} \circ \left( \left( t \circ 
			\overline{(g,i,j)} \right) \left( t \circ (1,f,1)\right) 
			\lambda_{t}(h,1,l)\right),
		\end{eqnarray*}
		where the last equality holds by Proposition \ref{charactanti}. Then
		\begin{eqnarray*}
			u&=& (g,i,j) \circ \overline{t} \circ \left( \left( t \circ 
			\overline{(g,i,j)}\right) \left( t \circ (h,f,l)\right)\right)\\
			&=& (g,i,j) \circ \overline{t} \circ \left(\left( t \circ 
			\overline{(g,i,j)}\right) \left( t \circ (g_2,f,q) (h,k,l) 
			(g_2^{-1},1,l)\right)\right)\\
			&=& (g,i,j) \circ \overline{t} \circ \left( \left( t \circ 
			(g,i,j)\right) \left( t \circ (\overline{t} (g,i,j)(g_2^{-1}, 
			1,l))\right)\right),
		\end{eqnarray*}
		where the second equality holds as $1_{\circ} B1_{\circ}$ is a left 
		brace 
		and 
		thus $(1_{\circ}B1_{\circ},\cdot)$ is abelian. So,
		\begin{eqnarray*}
			u&=& (g,i,j) \circ \overline{t} \circ t \circ \left( 
			\overline{(g,i,j)} 
			(g_2^{-1},1,l)\right)= (g,i,j) \circ (1,f,l).
		\end{eqnarray*}
		Note that if $a,b \in B$, then $\lambda_{a}(b) \circ \rho_{b}(a) = a 
		\circ \left( \overline{a}b\right) \circ 
		\left(\overline{\overline{a}b} \right)\circ b = a \circ b$. Hence, if 
		$r(a,b) = (w,x)$ it follows that $w \circ x =  a \circ b$. Then, 
		$(g,i,j) \circ (h,k,l) = u\circ v$. By the previous 
		calculations, it 
		follows that $u = (g,i,j) \circ (1,f,l)$. Thus, $v = \overline{(1,f,l)} 
		\circ 
		(h,k,l)$. Hence, 
		$$r_B^2\left( (g,i,j), (h,k,l) \right) = \left( (g,i,j) \circ (1,f,l), 
		\overline{(1,f,l)}\circ (h,k,l)\right).$$
		Note that $\lambda_{(g,i,j)}(h,k,l) \in 
		1_{\circ}B$ by Lemma \ref{behaviorlambda} and $\rho_{(h,k,l)}(g,i,j) 
		\in 
		B1_{\circ}$ by Lemma \ref{rhomap1}.Thus, $\lambda_{(g,i,j)}(h,k,l) = 
		(\alpha,1,p)$ and $\rho_{(h,k,l)}(g,i,j) = (\beta,z,1)$ for some 
		$\alpha, \beta, p, z$. Hence, applying the previous 
		formula to 
		$$r_B\left( (g,i,j), (h,k,l)\right) = \left( \lambda_{(g,i,j)}(h,k,l) , 
		\rho_{(h,k,l)}(g,i,j)\right),$$ we get that
		\begin{eqnarray*}
			r_B^3\left( (g,i,j), (h,k,l) \right) &=& r_B^2 \circ r\left( 
			(g,i,j), 
			(h,k,l) \right)\\ &=&r_B^2 \left( (\alpha, 1,p), (\beta, z,1) 
			\right)\\ &=& \left( (\alpha,1,p) \circ (1,1,1), \overline{(1,1,1)} 
			\circ (\beta, z, 1)\right)\\  &=& \left(\lambda_{(g,i,j)}(h,k,l), 
			\rho_{(h,k,l)}(g,i,j)\right).
		\end{eqnarray*}
		This shows that $r_B^3 = r_B$. 
	
	Suppose that $1_{\circ}B1_{\circ}$ is trivial. Then, in the notation of 
	above, we find 
	that
	\begin{eqnarray*}
	(1,i,j) \circ (1,f,l) &=& (1,i,j) \circ \left((1,f,q)(1,k,l)\right) \\
	&=& (1,1,1) \lambda_{(1,i,j)}(1,k,l) = \lambda_{(1,i,j)}(1,k,l).
	\end{eqnarray*}
	Moreover,
	\begin{eqnarray*}
	\rho_{(1,k,l)}(1,i,j) &=& \overline{ \overline{(1,1,j)} (1,k,l)} \circ 
	(1,k,l)\\ &=& \overline{ (1,f,q)(1,k,l)} \circ (1,k,l) \\ &=& \overline{ 
	(1,f,l)} \circ (1,k,l).
	\end{eqnarray*}
	Hence, with the above formula, it is clear that $r_B^2 = r_B$.
	\end{proof}

	\begin{definition}
		Let $r: X \times X \longrightarrow X \times X$ be a solution of the 
		Yang-Baxter equation. The monoid $$ M(r) := \left< x \in X 
		\mid 
		xy = uv \textnormal{ if } r(x,y) = (u,v) \right>$$
		is called the structure monoid of $r$. The multiplication in $M(r)$ 
		will be denoted by $*$.
	\end{definition}
	
	{\rm 	Let $(B,\cdot,\circ)$ be a left semi-brace such that its associated 
		$\rho$-map is an anti-homomorphism and let  $r_B$ be  its associated 
		solution. 
		Then, the structure monoid of $B$ is defined as $M(B)= M(r_B)$.
		Its identity element will be denoted by $1_{M(B)}$.
	To avoid confusion, we denote the product of $x$ and $y$ in
	$M(B)$ 
	as $x*y$.
	
	
	We first describe the monoid $M(B)$ in case the left semi-brace is a right 
	or 
	left zero left semi-brace.
	This will turn out to be very useful to describe the general case. Recall, 
	from 
	Theorem \ref{Colazzochar}, that 
	for a right zero left semi-brace $(E,\cdot,\circ)$ the semi-group 
	$(E,\cdot) 
	\cong \mathcal{M}\left(\left\lbrace 1 \right\rbrace, 
	1,J,\mathcal{I}_{|J|,1}\right)$.
	}

	\begin{theorem}
	 	Let $(E,\cdot,\circ)$ be a right zero  left semi-brace. Then, $$ 
		M(E) \cong \left(E \times \mathbb{N}\right)^1 =\left\lbrace (f,k) 
		\mid f \in E, k\geq 1 \textnormal{ a 
			positive integer}\right\rbrace \cup \left\lbrace 
		(1_{\circ},-1) \right\rbrace,$$
		where $(1_{\circ},-1)$ is the identity element and for any $(f,k),(g,l) 
		\in 
		\left( E 
		\times \mathbb{N}\right)^1$ the product $(f,k)(g,l)$ is defined as $(f 
		\circ g, k + l +1)$. In particular, $M(E)$ is contained in a 
		group. More specifically, this group is $E \times \mathbb{Z}$.
		
			Simiarly, if $(E,\cdot,\circ)$ be a left zero left semi-brace then, $$ 
		M(E) \cong \left(\mathbb{N} \times E \right)^1 = \left\lbrace 
		(k,f) \mid f \in E, k\geq 1 \textnormal{ a positive 
		integer}\right\rbrace.$$
		 In particular, $M(E)$ is contained in a group.

	\end{theorem}
	\begin{proof}
	         We prove the first part of the statement; the second part is 
	         proven analogously.
	        So assume $(E,\cdot,\circ)$ be a right zero  left semi-brace.
	        
		Every element $f \in E$ is an idempotent in $(E,\cdot)$. Hence, by 
		Lemma 
		\ref{behaviorlambda}, it follows that $r(e,f) = (e\circ f, 1).$
		Thus, in $M(E)$ we have that $e*f = (e \circ f)*1_{\circ}$. Because 
		$r(e,f) = \left(\lambda_{e}(f),\rho_{f}(e)\right) = (e\circ f, 
		1_{\circ})$, if $r(e,f) = (u,v)$, then $e \circ f = u \circ v$.  Hence, 
		if $f_1*\dotsb*f_n = y_1*\dotsb*y_n$, then $f_1 \circ \dotsb \circ f_n 
		= 
		y_1 
		\circ \dotsb \circ y_n$. Further,
		inductively applying $r$ on the last two non-trivial factors (i.e. 
		factors 
		which are not $1$), we get that
		\begin{align*}
		f_1*f_2* \dotsb *f_n &= f_1* f_2* \dotsb * (f_{n-1} \circ f_n) * 1\\ &= 
		\dotsb 
		\\&= 
		(f_1 \circ \dotsb \circ f_n)* \underbrace{1* \dotsb *1}_{n-1 
		\textnormal{ 
		times}}.
		\end{align*}
		So every non-trivial element can be uniquely written as $g*1^n$ for 
		some $g \in E$ 
		and $n\geq 0$.
		Define 
		the map $\varphi: \mathcal{M}(E) \longrightarrow \left(E \times 
		\mathbb{N}\right)^1\subset E \times \left< (1_{\circ},0)\right> 
		\subseteq E \times \mathbb{Z}$, by $\varphi(g.1^n) = 
		(g,n+1)$ and $\varphi(e) = (1_{\circ},-1)$. Clearly, $\varphi$ is a 
		monoid 
		isomorphism. The remaining claim follows immediately.
	\end{proof}

%
%

	\begin{lemma}\label{lemma1}
		Let $(B,\cdot,\circ)$ be a left semi-brace such that $\rho$ is an 
		anti-homo\-morphism. Then, 
		$\left(1_{\circ}B1_{\circ}\right)*\left(B1_{\circ}\right) = 
		\left(B1_{\circ}\right)*\left( 
		1_{\circ}B1_{\circ}\right)$ and 
		$\left(1_{\circ}B\right)*\left(1_{\circ}B1_{\circ}\right) = 
		\left(1_{\circ}B1_{\circ}\right)*\left(1_{\circ}B\right)$ in $M(B)$.
	\end{lemma}
	\begin{proof}
		Let us prove the first equality. Let $x \in B1_{\circ}$ and $y \in 
		1_{\circ}B1_{\circ}$. Then, $r(x,y) 
		\in \left(1_{\circ}B1_{\circ}\right)*\left( B1_{\circ}\right)$. Hence, 
		$\left(B1_{\circ}\right)*
		\left(1_{\circ}B1_{\circ}\right) 
		\subseteq \left(1_{\circ}B1_{\circ}\right)*\left( B1_{\circ}\right)$.
		Let $ g, h\in 1_{\circ}B1_{\circ}$ and $e \in E(B1_{\circ})$. We will 
		show that $g*\left(he\right)$ can also 
		be 
		written as an element of 
		$\left(B1_{\circ}\right)*\left(1_{\circ}B1_{\circ}\right)$. So we look 
		for $f \in E(B1_{\circ})$ and 
		$s,t 
		\in 1_{\circ}B1_{\circ}$ such that $$ g*\left( eh\right) = 
		\left(fs\right)* t.$$
		Take $f = \left( g \circ e \right) g^{-1}$, which is an element of 
		$E(B1_{\circ})$, 
		by Corollary \ref{GroupIdempotent}. Take $s = g \lambda_{g}(h) g^{-1}$ 
		and $ t 
		= 
		\lambda_{\overline{\left(g \circ (eh)\right) g^{-1}}}(g)$.
		We show that $r(fs,t) = (g,eh)$. Note that, by the semi-brace property, 
		$$ 
		fs = \left( g\circ e \right) g^{-1} g \lambda_g(h) g^{-1} = \left( 
		g\circ e 
		\right) \lambda_g(h) g^{-1} = \left( g \circ (eh) \right) g^{-1}.$$
		Hence, by lemma \ref{lambdaidempot},
		\begin{align*}
		\lambda_{fs}(t) &= \lambda_{\left( g \circ (eh) \right) g^{-1}} \left( 
		\lambda_{\overline{g \circ (eh) g^{-1}}}(g)\right) = 
		\lambda_{1_{\circ}}(g)		
		= g.
		\end{align*}
		Furthermore,
		\begin{align*}
		\rho_{t}(fs) = \left(\overline{\overline{fs}t}\right) \circ 
		\overline{fs} \circ g \circ (eh).
		\end{align*}
		So,
		\begin{align*}
		\overline{\rho_{t}(fs)} &= \overline{(eh)} \circ \overline{g} \circ 
		(fs) 
		\circ \left( \overline{fs} t\right)\\ &=\overline{(eh)} \circ 
		\overline{g} 
		\circ \left( \lambda_{fs}(t)\right) \\ &= \overline{(eh)} \circ 
		\overline{g}\circ g\\ &= \overline{(eh)}.
		\end{align*}
		Thus, indeed, $r(fs,t) = (g,eh)$ and thus $\left(fs\right)*t = 
		g*\left(eh\right)$ as desired.
		
		Let us now prove the second equality. 
		Using analogous reasoning as the first equality it follows, by Lemma 
		\ref{rhomap1}, that $\left(1_{\circ}B1_{\circ}\right)*1_{\circ}B 
		\subseteq 
		1_{\circ}B*\left(1_{\circ}B1_{\circ}\right)$. Let $g,h \in 
		1_{\circ}B1_{\circ}$ and 
		$e 
		\in 
		E(1_{\circ}B)$. Then we should show that there exist $t,s \in 
		1_{\circ}B1_{\circ}$ and $f \in 
		E(1_{\circ}B)$ such that $$ (ge) * h = t* (sf).$$
		Take $t = \overline{\rho_{\overline{ge}}\left(\overline{h}\right)}$, 
		$s= \lambda_{\overline{t}}(g)$ and $f= \lambda_{\overline{t}}(e)$. It 
		is left to the reader's discretion to check that $r(t,sf) = (ge,h)$.
	\end{proof}

\begin{lemma}\label{normalform}
	Let $(B,\cdot,\circ)$ be a left semi-brace such that $\rho$ is an 
	anti-homo\-morphism. Then, $M(B)$ is an  $\mathbb{N}$-graded 
	monoid. Moreover, \begin{eqnarray*} M(B) &=& \left\lbrace 1_{M(B)} 
		\right\rbrace \cup \underbrace{B}_{M(B)_1} \cup 
		\underbrace{\left(1_{\circ}B\right)*\left(B1_{\circ}\right)}_{M(B)_2} 
		\cup 
		\underbrace{\left(1_{\circ}B\right)*\left( 1_{\circ}B1_{\circ}\right)* 
			\left(B1_{\circ}\right)}_{M(B)_3} \cup\dotsb\\ && \cup 
		\underbrace{\left(1_{\circ}B\right)*\left(
			1_{\circ}B1_{\circ}\right)* \dotsb *\left(1_{\circ}B1_{\circ} 
			\right)* 
			\left(B1_{\circ}\right)}_{M(B)_n} \cup \dotsb\end{eqnarray*}
\end{lemma}
\begin{proof} 
	Let $x_1,\dotsb,x_n \in B$. By Lemma \ref{rhomap1},  $r(x_1,x_2) = \left( 
	\lambda_{x_1}(x_2), \rho_{x_2}(x_1)\right) \in 1_{\circ}B \times 
	B1_{\circ}$ and thus $x_1*x_2 \in \left( 1_{\circ}B*B1_{\circ}\right)$. 
	Hence, applying $r$ 
	from left to 
	right 
	$n-1$ times, we find $ x_1* x_2 * \dotsb * x_n \in 1_{\circ}B* 1_{\circ}B 
	\dotsb 
	B1_{\circ}$. Note that, by Lemma \ref{rhomap1} $\rho_{x}(y) \in 
	B1_{\circ}$ and by Theorem \ref{characttheorem} $B1_{\circ}$ is a 
	subsemi-brace. So, 
	$\rho_{x}(y) \in 1_{\circ}B1_{\circ}$ if $x,y \in 1_{\circ}B$. Thus, if 
	$n\geq 3$,
	applying $r$ from right to 
	left (starting with $r(x_{n-2},x_{n-1})$), it follows that $x_1*\dotsb*x_n 
	\in 
	\left(1_{\circ}B\right)* \left(1_{\circ}B1_{\circ}\right)* \dotsb 
	*\left(1_{\circ}B1_{\circ}\right)*\left(B1_{\circ}\right)$.
\end{proof}
	
	\begin{prop}\label{fingenmodule}
		Let $(B,\cdot,\circ)$ be a left semi-brace such that $\rho$ is 
		an 
		anti-homo\-morphism. Then, for any field $K$, the algebra $KM(B)$ is 
		generated as a left (and right) 
		$KM(1_{\circ}B1_{\circ})$-module by 
		$\left(1_{\circ}B\right)*\left(B1_{\circ}\right)$.
	\end{prop}
	\begin{proof}
		Applying Lemma \ref{normalform} and Lemma \ref{lemma1}, 
		it follows that $KM(B)$ is a finitely generated 
		$KM(1_{\circ}B1_{\circ})$-module with generators $\left\lbrace 
		1_{M(B)} 
		\right\rbrace \cup B \cup 1_{\circ}B.B1_{\circ}$.
	\end{proof}

\begin{theorem}\label{lebed}	 (Lebed and Vendramin \cite[Remark 5.14]{lebed2017structure}) 
\label{GKdim}
If  $B$ is a finite skew left brace then $KM(B)$ is a finite module over an algebra $KA$ where $A$ is an abelian monoid generated by at most $|B|$ elements.
In particular, $KM(B)$ is a Noetherian PI-algebra of Gelfand-Kirillov dimension at most $|B|$.

\end{theorem}

\begin{theorem}\label{maintheorem}
	Let $(B,\cdot,\circ)$ be a finite left semi-brace such that $\rho$ is an 
	anti-homomorphism. 
	Then, $KM(B)$ is a Noetherian, PI-algebra of finite 
	Gelfand-Kirillov dimension equal to that of $KM(1_{\circ}B1_{\circ})$. In 
	particular, this dimension is at most $|1_{\circ}B1_{\circ}|$ and it is 
	precisely equal to $|1_{\circ}B1_{\circ}|$ if $B$ is a left
	brace.
\end{theorem}
\begin{proof}
	By Proposition \ref{fingenmodule}, $KM(B)$ is both a finitely 
	generated left and right $KM(1_{\circ}B1_{\circ})$-module. 
	By Corollary \ref{rowcol} the left subsemi-brace $1_{\circ}B1_{\circ}$ is a 
	skew left brace.  So, because of Theorem~\ref{GKdim}, $KM(1_{\circ}B1_{\circ})$ 
	is a Noetherian PI-algebra of Gelfand-Krillov dimension at most $|1_{\circ}B1_{\circ}|$.
	Well-known results then show that $KM(B)$ inherits these properties (see for example \cite[8.2.9]{noncommutativenoetherianrings}
	and \cite[5.1.6]{jespers2007noetherian}).
%
\end{proof}

\begin{corollary}
	Let $(B,\cdot,\circ)$ be a completely simple left semi-brace such that 
	$\rho$ is an anti-homomorphism.
	Then the group algebra $KG(B)$ is a Noetherian PI-algebra of finite 
	Gelfand-Kirrilov dimension. Note that $M(B)$ is not embedded in $G(B)$ in 
	general.
\end{corollary}

\bibliographystyle{plain}
{\footnotesize \bibliography{Bracesbib2}}

\end{document}